\documentclass[final]{siamart220329}  % ,onetabnum
\usepackage{graphicx}
\usepackage{mathptmx}
\usepackage{hyperref}
\usepackage[english]{babel}
\usepackage[a4paper,left=25mm,right=25mm,top=25mm,bottom=25mm]{geometry}
\usepackage{pagecolor}
\usepackage{amsmath,amsfonts,latexsym,amscd,amssymb,theorem}
\usepackage{epsfig}
\usepackage{color}
\usepackage{amsmath}
\usepackage{longtable}
\usepackage{booktabs}
\usepackage{graphicx}
\usepackage{bbm}
\usepackage[T1]{fontenc}
\usepackage[utf8]{inputenc}
\usepackage{esvect}
\usepackage{multirow}
\usepackage{caption,subcaption}
\usepackage{kpfonts}
\usepackage{tikz}
\usepackage{babelbib}

\tikzset{
  treenode/.style = {shape=rectangle, rounded corners,
                     draw, align=center,
                     top color=white, bottom color=blue!20},
  root/.style     = {treenode, font=\Large, bottom color=red!30},
  env/.style      = {treenode, font=\ttfamily\normalsize},
  dummy/.style    = {circle,draw}
}

\usepackage{algorithm}
\usepackage{algpseudocode}

\newcommand{\R}{\mathbb{R}\,}

\headers{Data-driven Tensor Train Gradient Cross Approximation for HJB Equations}{S. Dolgov, D. Kalise and L. Saluzzi}

\title{Data-driven Tensor Train Gradient Cross Approximation for Hamilton-Jacobi-Bellman Equations}

\author{
Sergey Dolgov\thanks{Department of Mathematical Sciences, University of Bath, North Rd, BA2 7AY Bath, United Kingdom ({\tt sd901@bath.ac.uk})}
\and
Dante Kalise\thanks{Department of Mathematics, Imperial College London, South Kensington Campus, SW7 2AZ London, United Kingdom ({\tt dkaliseb@imperial.ac.uk})}
\and
Luca Saluzzi\thanks{Department of Mathematics, Imperial College London, South Kensington Campus, SW7 2AZ London, United Kingdom ({\tt lsaluzzi@ic.ac.uk})}}

\begin{document}

\maketitle

\begin{abstract}
A gradient-enhanced functional tensor train cross approximation method for the resolution of the Hamilton-Jacobi-Bellman (HJB) equations associated to optimal feedback control of nonlinear dynamics is presented.
The procedure uses samples of both the solution of the HJB equation and its gradient to obtain a tensor train approximation of the value function. The collection of the data for the algorithm is based on two possible techniques: Pontryagin Maximum Principle and State-Dependent Riccati Equations. Several numerical tests are presented in low and high dimension showing the effectiveness of the proposed method and its robustness with respect to inexact data evaluations, provided by the gradient information.
The resulting tensor train approximation paves the way towards fast synthesis of the control signal in real-time applications.
\end{abstract}
\begin{keywords}
Dynamic Programming, Optimal Feedback Control, Hamilton-Jacobi-Bellman
Equations, Tensor Decomposition, High-dimensional Approximation
\end{keywords}
\begin{AMS}
15A69, 15A23, 65F10, 65N22, 49J20, 49LXX, 49MXX
\end{AMS}

\section{Introduction}

Stabilization of nonlinear dynamical systems is a fundamental problem in control theory, with applications in mechanical systems, chemical engineering, and fluid flow control, among many other areas. Nonlinear stabilization is often approached by means of feedback (closed-loop) controllers which, in contrast to open-loop controls, offer enhanced stability properties with respect to external disturbances. The synthesis of optimal feedback controls resorts to the use of dynamic programming, which characterizes the optimal feedback law in terms of the solution of a Hamilton-Jacobi-Bellman (HJB) nonlinear Partial Differential Equation (PDE). 
The main drawback for this approach lies on the fact that the HJB equation must be solved on the state space of the dynamical system, often leading to solving a nonlinear PDE in arbitrarily high dimensions. This limitation is referred to as the \textit{curse of dimensionality}, a term coined by Richard Bellman in the '60s and still an active subject of research.
Under some specific structural assumptions, as in the case of linear dynamics and a quadratic cost functional, the HJB equation is equivalent to the matrix Algebraic Riccati Equation (ARE), for which high-dimensional solvers are readily available \cite{Kirsten_Simoncini_2020,
BBKS_2020}. Unfortunately, for the fully nonlinear setting, no reformulation is possible and the HJB PDE must be solved directly. In this direction, over the last years there has been a significant progress on the solution of high-dimensional HJB PDEs arising in optimal control, including max-plus algebra methods \cite{McEneaney_2007,Akian_Gaubert_Lakhoua_2008,maxplusdarbon}, sparse grids \cite{GK16}, tree-structure algorithms \cite{Alla_Saluzzi_2020,falcone2023approximation}, applications of artificial neural networks \cite{Han_Jentzen_E_2018,Darbon_Langlois_Meng_2020,Kunisch_Walter_2021,sympocnet,Zhou_2021,Onken2021,ruthotto2020machine} and regression-type methods
in tensor formats \cite{sallandphd,richter2021solving}.
The above mentioned techniques can scale up to very high-dimensional HJB PDEs, however, the effective implementation of real-time HJB-based controllers remains an open problem. 

In this work we develop a data-driven approach based on the knowledge of the value function and its gradient on sample points. Similar ideas have been proposed in \cite{Kang_Wilcox_2017,Kang_Wilcox_2015} in the framework of sparse grids, in \cite{Azmi:2021} with sparse polynomial regression, in \cite{oss-hjbt-2021,sallandphd} via tensor train representation and Monte Carlo quadrature, and in \cite{Nakamura_Zimmerer_2021,Nakamura_Zimmerer_2021b,bensnn,onken2021neural} using supervised learning and deep neural networks. The aforementioned works exploit the link between the HJB equation and Pontraygin's Maximum Principle (PMP), a first-order optimality condition which is interpreted as a characteristic curve of the HJB PDE. The latter is used to generate synthetic data for the solution of the HJB equation, whose global solution is then recovered by supervised learning. A similar idea is proposed in \cite{ABK21} where the authors propose a sub-optimal feedback law obtained via a feedforward neural network. In this case the training set is generated via the State-Dependent Riccati Equation (SDRE) strategy \cite{Banks_Lewis_Tran_2007,609663,allasdre}, an extension of the Riccati solution to nonlinear dynamics.

Generally speaking, the proposed methodology belongs to a class of surrogate models, where instead of solving numerically expensive PMP and SDRE problems for each given state of the system, an \emph{offline} approximation of the entire value function is pre-computed in a compressed storage format.
This format is then used for a cheap \emph{online} evaluation of an approximate control signal for any given state of the system, which allows one to control the system in real time even on a low-performance device (e.g. FPGA).
In this paper we propose to approximate the value function together with its gradient in a tensor product decomposition, computed via adaptive sampling of either PMP or SDRE.
We show that sampling from this pre-built tensor format of the value function is 100 times faster than the online computation of SDRE solutions, effectively enabling real-time control synthesis.

Tensor decompositions have emerged as efficient approximation techniques for high dimensional tensors \cite{hackbusch-2012,larskres-survey-2013},
and, when such tensors encapsulate expansion coefficients of functions in a structured basis, multivariate functions \cite{Marzouk-stt-2016,Gorodetsky-ctt-2019}.
The idea behind tensor decompositions is to approximate a given tensor (or a function) by separation of variables.
Convergence of such decompositions for functions of certain regularity \cite{uschmajew-approx-rate-2013}, or particular classes of value functions \cite{DKK21} has been established.
Hierarchical tensor decompositions \cite{hackbusch-2012}, in particular the Tensor Train decomposition \cite{osel-tt-2011}, have become most widely used due to efficient and numerically stable computational algorithms.
Those can typically be written in a recursive form that scales linearly with the dimension.
The main workhorse is the Alternating Linear Scheme \cite{holtz-ALS-DMRG-2012,ds-amen-2014}, an analog of the coordinate gradient descent, which optimizes a desired objective function by iterating over the tensor decomposition factors, computing only one factor at a time.

This optimization framework can be used for solving regression problems, such as the Variational Monte Carlo \cite{oss-hjb0-2019,Eigel-VMC-2019,oss-hjbt-2021,foss-exittime-2020}.
In this case, one draws random samples from the sought function, and seeks for its tensor approximation by minimizing the misfit on the given samples over the elements of the tensor factors.
This problem is also called Tensor Completion \cite{KSV-completion-2014,kramer-SALSA-2019}.
In addition to the straightforward coordinate descent,
one can formulate the misfit optimization as a gradient flow on a Riemannian manifold of the tensor decomposition \cite{Steinlechner-completion-2016}, which turns out to be more accurate in a higher-error regime.

Although Tensor Completion works well for smooth functions and rapidly converging decompositions, the predefined sampling may miss localized but significant regions of a more irregular function.
Cross approximation methods \cite{ot-ttcross-2010,so-dmrgi-2011proc,grasedyck-par-cross-2015,sav-qott-2014} have been designed to adapt the sampling sets to minimize the conditioning of the interpolation problem, and to improve the approximation accuracy.
Moreover, the structure of the sampling sets in the cross methods is aligned to the structure of the sought tensor decomposition, which enables a more efficient linear algebra.

In addition to optimizing the locations of the data samples, one can assimilate more information per each sample.
In some problems (such as the optimal control considered here),
each evaluation of the sought function comes together with a value of the gradient of this function at little or no extra cost.
In this case, one can extend the regression problem such that weighted misfits in both the function and its gradient are minimized.
This allows one to take fewer samples for the same accuracy \cite{ro-grad-2018,ABK21}, or to achieve a higher accuracy for the same amount of samples.
The latter property becomes especially useful when the function values are noisy.
In this paper we develop an algorithm to solve the gradient-enhanced regression problem on a Tensor Train decomposition of the value function of an optimal control problem.

The contributions of this paper can be summarized as follows:
\begin{enumerate}
    \item We propose a framework for synthetic data generation for infinite horizon nonlinear stabilization problems based on the pointwise solution of State-Dependent Riccati Equations.
    \item We formulate a gradient-augmented supervised learning problem where a tensor train approximation of the value function is adaptively trained upon synthetic SDRE samples.
    \item We develop a two-box approach (based on the two ingredients above) to improve the accuracy in cases where stabilization towards the origin requires enhanced precision of the control law.
    \item We present a comprehensive computational assessment of the proposed methodology over high-dimensional nonlinear tests motivated by optimal control of multi-agent systems.
\end{enumerate}

The rest of the paper is structured as follows. In Section \ref{sec:iocp} we give a brief introduction on the infinite horizon optimal control problems and on the two techniques used to generate the dataset: the State-Dependent Riccati Equations and the finite horizon Pontryagin Maximum Principle. In Section \ref{sec:gc} we develop the construction of Gradient Cross for the approximation of the value function. Finally, in Section \ref{sec:nt} we will demonstrate the efficiency of the proposed algorithm in low and high dimensional numerical tests.

\section{The infinite horizon optimal control problem and suboptimal feedback laws}\label{sec:iocp}
In this section we formulate the deterministic infinite horizon problem for which we are interested in synthesizing a feedback law. We first present the optimal feedback synthesis using the Hamilton-Jacobi-Bellman formalism, to subsequently discuss suboptimal feedback laws which can be effectively cast in a data-driven environment. We consider system dynamics in control affine-form given by
\begin{equation}\label{eq}
\left\{ \begin{array}{l}
\dot{y}(s)=f(y(s)) + B(y(s))u(s), \;\; s\in(0,+\infty),\\
y(0)=x\in\mathbb{R}^d.
\end{array} \right.
\end{equation}
We denote by $y:[0,+\infty)\rightarrow\R^d$ the state of the system, by $u:[0,+\infty)\rightarrow\R^m$ the control signal and by $\mathcal{U}=L^\infty ([0,+\infty);U)$ the set of admissible controls where $U\subset \R^m$ is a compact set. We assume that the system dynamics $f:\R^d\rightarrow\R^d$ and $B:\R^d\rightarrow\R^d$ are $\mathcal{C}^1(\R^d)$ functions, verifying $f(\underline{0})=B(\underline{0})=\underline{0}$. Whenever we want to stress the dependence of the control signal from an initial state $x \in \mathbb{R}^d$, we will write $u(t,x)$. 

\noindent We consider the following  undiscounted infinite horizon cost functional:
\begin{equation}\label{cost}
 J(u(\cdot,x)):=\int_0^{+\infty} y(s)^{\top} Q y(s)+ u^{\top}(s) R u(s)\,ds\,,
\end{equation}
where $Q\in \R^{n\times n}$ is a symmetric positive semidefinite matrix and $R\in \R^{m\times m}$ is a symmetric positive definite matrix. Our goal is to synthesize an optimal control in feedback form, that is,  a control law that is fully determined upon the current state of the system. We begin by the defining the value function for a given initial condition $x \in\R^d$:
\begin{equation}
V(x):=\inf\limits_{u\in\mathcal{U}} J(u(\cdot,x))\,,
\label{VF}
\end{equation}
which, by standard dynamic programming arguments, satisfies the following Hamilton-Jacobi-Bellman PDE for every $x\in\R^d$
\begin{equation}\label{HJB}
\min\limits_{u\in U }\left\{(f(x)+B(x)u)^{\top} \nabla V(x) + x^{\top} Q x+ u^{\top} R u \right\}=0.
\end{equation}
%where $DV(x)$ stands for the gradient of the value function.
The HJB PDE \eqref{HJB} is challenging first-order fully nonlinear PDE cast over $\R^d$, where $d$ can be arbitrarily large, and thus intractable through conventional grid-based methods. However, in the unconstrained case, $i.e.$ $U=\R^m$, the minimizer of the l.h.s. of eq. \eqref{HJB} can be computed explicitly as
\begin{equation}\label{optc}
u^*(x)=-\frac{1}{2} R^{-1}B(x)^{\top} \nabla V(x)\,,
\end{equation}
leading to an unconstrained version of the HJB PDE given by
	\begin{align}\label{hjbu}
	\nabla V(x)^{\top}f(x)-\frac14\nabla V(x)^{\top}B(x)R^{-1}B(x)^{\top}\!\nabla V(x)\!+\!x^{\top}\!Qx=0\,.
\end{align}

In this work, we are interested in recovering an approximation of the optimal feedback law \eqref{optc} circumventing the solution of the high-dimensional HJB PDE \eqref{hjbu}. Instead, we will approximate $V(x)$ in a regression framework, assuming measurements from both the value function $V(x)$ and its gradient $\nabla V(x)$ are available at sampling points. The idea behind this approach resides in the fact that, for a given sample state $x$, the value function and its gradient can be recovered by minimizing the cost \eqref{cost} without resorting the HJB PDE, and independently from other sampling points. However, generating a  single sample  of $V(x)$ by solving the associated infinite horizon optimal control problem \eqref{cost} is already a computationally demanding task, unsuitable for the generation of a large-scale dataset. Instead, we propose two alternatives which trade optimality in minimizing \eqref{cost} by fast computability: a State-Dependent Riccati Equation (SDRE) approach and a finite horizon Pontryagin Maximum Principle (PMP) formulation. In this way, we generate a synthetic dataset which approximates the value function and its gradient, leading to a suboptimal, yet asymptotically stabilizing feedback law. 

\subsection{State-Dependent Riccati Equation}

Since the value function is a positive function, with no loss of generality it can be represented as
\begin{equation}
V(x)=x^{\top} \Pi(x) x,
\label{Vsdre}
\end{equation}
where $\Pi(x) \in \mathbb{R}^{n \times n}$ is a symmetric matrix-valued function with its gradient given by the following formula
\begin{equation}
\nabla V(x)=2  \Pi(x) x + 
\begin{pmatrix}
    x^{\top} \frac{d}{dx_1}\Pi(x) x\\
    \vdots\\
     x^{\top} \frac{d}{dx_d}\Pi(x) x
  \end{pmatrix} .
  \label{Dsdre}
 \end{equation}
In the particular case of a linear dynamics ($i.e.$ $A(x)=A$ and $B(x)=B$),  $\Pi(x)=\Pi$ is constant and it is well known that the HJB equation \eqref{hjbu} becomes the Algebraic Riccati Equation (ARE)
$$
A^{\top} \Pi + \Pi A - \Pi B R^{-1} B^{\top} \Pi+Q=0 .
$$
An intermediate parametrization can be considered in the nonlinear case, by writing the dynamics in semilinear form
\begin{equation}
\dot{y}= A(y(t)) y(t) +B(y(t)) u(t),
\end{equation}
In this case, eq. \eqref{hjbu} can be approximated as
\begin{equation}
A^{\top}(x) \Pi(x) + \Pi(x) A(x) - \Pi(x) B(x) R^{-1} B(x)^{\top} \Pi(x)+Q=0 \,,
\label{SDRE}
\end{equation}
which is obtained by applying the ansatz $V(x)=x^{\top} \Pi(x) x$ with a gradient approximation $\nabla V(x)=2  \Pi(x) x$, that is, by neglecting the second term in eq. \eqref{Dsdre} . The resulting equation is known as State-Dependent Riccati Equation (SDRE).
First, we note that the SDRE is a functional equation which must hold for every $x \in \mathbb{R}^d$, hence analytical solutions are only available in limited cases. However, under the assumption that the pair $(A(x),B(x))$ is stabilizable for all $x\in \mathbb{R}^d$, in \cite{Banks_Lewis_Tran_2007} the authors prove that the feedback law associated to the SDRE
\begin{equation}
u(x)=-R^{-1} B(x)^{\top} \Pi(x) x
\label{opt_control}
\end{equation}
is locally asymptotically stabilizing (that is, for states in a neighborhood of the origin). Since the SDRE is an approximation of the HJB PDE leading to an optimal feedback law, we claim that the SDRE control is a suboptimal feedback law. 

A natural implementation of the SDRE control law for nonlinear stabilization is through a receding horizon approach. In the SDRE setting, this means that given a current state $x^k$ of the trajectory, every matrix in eq.\eqref{SDRE} is frozen at $x^k$ and an algebraic Riccati equation is solved for $\Pi(x^k)$. Then, the resulting feedback law $u(x^k)$ from \eqref{opt_control} is applied to evolve the dynamics for a short horizon, until the next state $x^{k+1}$ where the computation is repeated. This implementation is feasible for low-dimensional dynamics, but becomes quickly unpractical as the number of states grows, as it requires the solution of large-scale algebraic Riccati equations at a very high rate. Here instead, we propose a supervised learning approach where eq. \eqref{SDRE} is used to generate a dataset to approximate $\Pi(x)$ offline, so that online feedback calculations are limited to the evaluation of the feedback law \eqref{opt_control}, which requires the evaluation of $\Pi(x)x$, or the approximate gradient of $V(x)$. The approximation of the value function using a functional tensor train format is discussed in detail in Section \ref{sec:ftt}.

When  approximating the value function via supervised learning, as presented in Section \ref{sec:gc}, we will augment our regression dataset with values of both $V(x)$ and its gradient. This computation relies on the formula \eqref{Dsdre}, which requires derivatives of the SDRE solution $\Pi(x)$.
Without loss of generality let us consider the case $B(x)=B$. Computing derivatives of \eqref{SDRE} with respect to a generic coordinate $x_i$ and denoting by $W=BR^{-1}B^{\top}$, we obtain the following Lyapunov equation for $\frac{d}{dx_i} \Pi(x)$
\begin{equation}
\frac{d}{dx_i} \Pi(x) \left(A(x)-W \Pi(x) \right) + \left(A(x)^{\top}-\Pi(x) W \right) \frac{d}{dx_i} \Pi(x) = -\frac{d}{dx_i} A(x)^{\top} \Pi(x) - \Pi(x) \frac{d}{dx_i} A(x), \quad i=1,\ldots, d.
\label{dSDRE}
\end{equation}
Therefore, for each sampled state $x$, the computation of $V(x)$ and its gradient requires the solution of one ARE (freezing $x$ in \eqref{SDRE}) and $d$ Lyapunov equations \eqref{dSDRE}, where $d$ is the dimension of the dynamical system.
If the matrix $A(x)$ does not depend on a variable $x_i$, its derivative with respect to that variable will be a null matrix and by \eqref{dSDRE} also the matrix $\frac{d}{dx_i} \Pi(x) = 0_d$, will be null. Hence, it will not contribute in the computation of the gradient of the value function. It will be sufficient to solve equation \eqref{dSDRE} just $k$ times, where $k$ is the number of variables appearing in $A(x)$.

\subsection{Pontryagin Maximum Principle}
Another option to generate an approximation of the infinite horizon value function is the use of Pontryagin's Maximum Principle (PMP). We refer to Chapter 5.3 in \cite{Kirk2004} for a complete description of this methodology. However, PMP provides first-order optimality conditions for a finite horizon optimal control problem. Since we are dealing with an infinite horizon, it is necessary to provide a final time $T$ for which the value function and its derivative decay to zero for every initial condition chosen in a reference domain.In this framework, the suboptimality originates from the truncation of the infinite horizon.For an initial condition $x$ and time horizon $T$,  introducing the adjoint variable $p:[0,T]\rightarrow\R^d$, the PMP system reads for \eqref{cost}
\newpage
\begin{equation}\label{pmp}
\left\{ \begin{array}{l}
\frac{d}{dt}y^*(t)=f(y(t)) + B(y^*(t))u^*(t), \\
y^*(0)=x,\\
-\frac{d}{dt}p_i^*(t)= \sum_{j=1}^d p^*_j(t) \partial_{y_i} (f_j(y^*(t))+B_j(y^*(t))u^*(t))+2(Qy^*)_i, \;\; i=1,\ldots,d,\\
p_i(T)=0 , \\
u^*(t)=-\frac{1}{2} R^{-1} B(y^*(t))^{\top} p^*(t) \,.
\end{array} \right.
\end{equation}

In \cite{Subbotina2006} the author shows that the optimal adjoint corresponds to the gradient of the value function along the optimal trajectory. Hence, the value function on the initial condition $x$ will be computed along the optimal trajectory and the optimal control
\begin{equation}
V(x)= \int_0^{+\infty} y^*(s)^{\top} Q y^*(s) + u^*(s)^{\top} R u^*(s) \, ds \approx \int_0^{\top} y^*(s)^{\top} Q y^*(s) + u^*(s)^{\top} R u^*(s) \, ds ,
\label{Vpmp}
\end{equation}
while its  gradient will be given by initial value of the adjoint, $i.e.$
\begin{equation}
\nabla V(x)\approx p(0).
\label{DVpmp}
\end{equation}

Equality \eqref{Vpmp} holds with the assumption that the optimal control $u^*(s)$ and the optimal trajectory $y^*(s)$ reached the zero level in the time interval $[0,T]$.
%Since the optimal control is given by \eqref{opt_control} and $p(t)=\nabla V(x(t))$ {\color{red} needs to be made more explicit using the control and adjoint, depends on what is your additional optimality condition in the PMP above}, the PMP system \eqref{pmp} can be rewritten as a two-point boundary value problem for the optimal pair $(y^*(t),p^*(t))$
%\begin{equation}\label{pmp2}
%\left\{ \begin{array}{l}
%\frac{d}{dt}y^*(t)=f(y^*(t)) -\frac{1}{2}B(y^*(t))R^{-1}B(y^*(t))^{\top} p^*(t), \\
%y^*(0)=x,\\
%-\frac{d}{dt}p_i^*(t)= \sum_{j=1}^n p^*_j(t) (\partial_{y_j} f_j(y^*(t)) -\frac{1}{2}\partial_{y_j}\left(B_j(y^*(t))R^{-1}B(y^*(t))^{\top} \right) p^*(t))+2(Qy^*(t))_i, \;\; i=1,\ldots,d,\\
%p_i(T)=0\,.
%\end{array} \right.
%\end{equation}

\subsubsection{Control constrained problem}
\label{constraint}
Unlike the SDRE approach, using PMP enables the addition of control constraints. For the sake of simplicity, we discuss the scalar control case, i.e. $m=1$, with a control signal restricted to an interval $[-u_{max},u_{max}]$. We are going to consider the same approach used in \cite{DKK21}. In order to impose the control constraints, let us choose the following penalty function 
$$
\mathcal{P}(u)=u_{max} \tanh(u/u_{max})
$$
 and let us change the control penalty term $u^{\top}Ru$ in the cost functional \eqref{cost} with the following term
\begin{equation}
W(u)= 2 R  \int_0^u \mathcal{P}^{-1}(\mu)\, d \mu.
\label{Wu}
\end{equation}

Note that the choice of this penalty function will keep the control in the interval $[-u_{max},u_{max}]$.

The corresponding PMP system becomes
\begin{equation}\label{pmpP}
\left\{ \begin{array}{l}
\frac{d}{dt}y^*(t)=f(y(t)) + B(y^*(t))\mathcal{P}(u^*(t)), \\
y^*(0)=x,\\
-\frac{d}{dt}p_i^*(t)= \sum_{j=1}^n p^*_j(t) \partial_{y_j} (f_j(y^*(t))+B_j(y^*(t))\mathcal{P}(u^*(t)))+2(Qy^*)_i, \;\; i=1,\ldots,d,\\
p_i(T)=0 , \\
u^*(t) =-\frac{1}{2} R^{-1}B(y^*(t))^{\top} p^*(t).
\end{array} \right.
\end{equation}

In this case we need to compute the value function using the control cost functional \eqref{Wu}, leading to the following formula

\begin{equation}
V(x)= \int_0^{\top} y^*(s)^{\top} Q y^*(s) + W(u^*(s)) \, ds ,
\label{cost_const}
\end{equation}
while the gradient is still obtained by the initial value of the adjoint $p(0)$.

\subsection{Functional Tensor Train}
\label{sec:ftt}
The value function defined by \eqref{VF} lives in the same dimension of the dynamical system \eqref{eq}. We are interested in dealing with high dimensional dynamical systems, $e.g.$ those deriving from the semidiscretization of PDEs or from complex systems. For this reason we need an efficient procedure to deal with high-dimensional functions. We are going to consider the Functional Tensor Train (FTT) to mitigate this problem. We sketch the main aspects, further details can be found in \cite{Marzouk-stt-2016,Gorodetsky-ctt-2019}.

First of all, let us fix for each variable $x_k$ a set of $n_k$ basis functions $\Phi_k(x_k):=\{\Phi^{(1)}_k(x_k),\dots , \Phi^{(n_k)}_k(x_k) \}$ and a set of collocation points $X_k=\{x_k^{(i)}\}_{i=1}^{n_k}$.
For uniqueness of the representation we assume that the Vandermonde matrix $\Phi_k(X_k) \in \mathbb{R}^{n_k\times n_k}$ is nonsingular.
Classical choices for the basis functions are the Lagrange basis or the Legendre polynomials.
Now, given a multivariate function $V:X := \varprod_{k=1}^d X_k  \rightarrow \mathbb{R}$, we are interested in the following approximation, which we call FTT:
\begin{equation}
V(x) \approx \tilde{V}(x) := \sum_{\alpha_0=1}^{r_0} \sum_{\alpha_1=1}^{r_1} \, \dotsi \,\sum_{\alpha_d=1}^{r_d} G^{(1)}_{(\alpha_0,\alpha_1)}(x_1) \, \dotsi \, G^{(k)}_{(\alpha_{k-1},\alpha_k)}(x_k) \,\dotsi\, G^{(d)}_{(\alpha_{d-1},\alpha_d)}(x_d).
\label{FTTd}
\end{equation}
The summation ranges $r_k$ are called \emph{TT ranks} and the factor $G^{(k)}_{(\alpha_{k-1},\alpha_k)}(x_k)$ is called $k$-th \emph{TT core}.
Without loss of generality we can fix $r_0=r_d=1$.
The TT core is a linear combination of the $n_k$ basis functions:
$$
G^{(k)}_{(\alpha_{k-1},\alpha_k)}(x_k)= \sum_{i=1}^{n_k} \Phi^{(i)}_k(x_k) H^{(k)}_{(\alpha_{k-1},i,\alpha_k)}=\Phi_k(x_k)\cdot H^{(k)}_{(\alpha_{k-1},\alpha_k)}\;,
$$
where $H^{(k)} \in \mathbb{R}^{r_{k-1}\times n_k \times r_k}$ is a three-dimensional tensor,
$H^{(k)}_{(\alpha_{k-1},i,\alpha_k)} \in \mathbb{R}$ is its element, and
$H^{(k)}_{(\alpha_{k-1},\alpha_k)} \in \mathbb{R}^{n_k}$ is a vector from $H^{(k)}$, sliced at the given indices $\alpha_{k-1},\alpha_k$.
Similarly, we introduce a matrix slice $H^{(k)}_{(i)} \in \mathbb{R}^{r_{k-1} \times r_k}$ and a matrix valued function $G^{(k)}(x_k): X_k \rightarrow \mathbb{R}^{r_{k-1} \times r_k}$.
This allows us to ease the notation, and write \eqref{FTTd} in a matrix form as follows
$$
\tilde{V}(x) =  G^{(1)}(x_1)\, \dotsi \,G^{(k)}(x_k)\, \dotsi \, G^{(d)}(x_d).
$$ 
In the following sections we are going to refer to the TT rank of a tensor as the maximum among all the TT ranks, $r=\max_{k=0,\ldots,d} r_k$.

The TT decomposition was initially written for discrete tensors~\cite{osel-tt-2011}.
A relation to function approximation is established via the same Cartesian basis,
$$
\tilde V(x) = \sum_{i_1,\ldots,i_d=1}^{n_1,\ldots,n_d} H_{(i_1,\ldots,i_d)} \Phi^{(i_1)}_1(x_1) \cdots \Phi^{(i_d)}_d(x_d).
$$
This corresponds to the TT decomposition
\begin{equation}\label{eq:ttd}
H_{(i_1,\ldots,i_d)} = \sum_{\alpha_0,\ldots,\alpha_d=1}^{r_0,\ldots,r_d} H^{(1)}_{(\alpha_0,i_1,\alpha_1)} \cdots H^{(d)}_{(\alpha_{d-1},i_d,\alpha_d)}.
\end{equation}
Counting the number of elements in the tensors in the right hand side, we notice that the TT decomposition needs $\sum_{k} r_{k-1} n_k r_k = \mathcal{O}(dnr^2)$ degrees of freedom (where we introduce $n:=\max_k n_k$), in contrast to $\mathcal{O}(n^d)$ elements in the full tensor of coefficients $H$.
Other tensor decompositions can be used, such as the HT format \cite{hk-ht-2009,hackbusch-2012} or the range-separated tensor format \cite{bkk-range-sep-2018}. However, in this paper we prefer to use the TT format as the most simple yet general representation, which is suitable for value functions as they are usually smooth.

The TT format admits fast linear algebraic operations.
For instance, we can compute multivariate integrals of $\tilde V(x)$ by using a tensor product of high-order univariate quadratures.
Let $\{x_k^{(i)}\}$ and $\{w_k^{(i)}\}$ be quadrature nodes and weights respectively.
Then the integral $\int V(x) dx$ can be approximated by
$$
\left(\cdots \left(\left[\sum_{i_1=1}^{n_1} w_1^{(i_1)} G^{(1)}(x_1^{(i_1)})  \right] \cdot \left[\sum_{i_2=1}^{n_2} w_2^{(i_2)} G^{(2)}(x_2^{(i_2)})\right] \right)\cdots \right) \cdot \left[\sum_{i_d=1}^{n_d} w_d^{(i_d)} G^{(d)}(x_d^{(i_d)}) \right],
$$
requiring $\mathcal{O}(dnr^2)$ operations in this order.
Similarly, we can compute derivatives, and approximate
\begin{equation}
\nabla_{x_k} V(x) \approx G^{(1)}(x_1)\cdots G^{(k-1)}(x_{k-1}) \cdot \left[\frac{d}{dx_k} G^{(k)}(x_k) \right] \cdot G^{(k+1)}(x_{k+1}) \cdots G^{(d)}(x_d),
\label{grad_ftt}
\end{equation}
again needing $\mathcal{O}(dnr^2)$ operations per point.
Additions, inner and pointwise products, as well as actions of linear operators can be written as explicit TT decompositions with a cost that is linear in $d$.
Such explicit decompositions are likely to have overestimated TT ranks.
However, if a tensor \eqref{eq:ttd} (or a function \eqref{FTTd}) is given in a TT format, a quasi-optimal re-approximation can be done in $\mathcal{O}(dnr^3)$ operations by using QR and SVD factorizations.
For details we refer to~\cite{osel-tt-2011}.

Explicit TT formats are a rare exception though.
In general, a function may have no exact decomposition, and an approximation must be sought from (as few as possible) evaluations of the function.
This can be achieved by solving a Least Squares problem, by minimizing the sum of squares of errors at given (e.g. random) points over the elements of TT cores \cite{oss-hjb0-2019,Eigel-VMC-2019,oss-hjbt-2021,foss-exittime-2020}.
However, a priori chosen point sets may miss an important region of the domain, which will make the approximation inaccurate.
Alternatively, \emph{TT-Cross} methods \cite{ot-ttcross-2010,so-dmrgi-2011proc,sav-qott-2014} adapt the sampling points iteratively towards the optimal locations for the current iterate.
However, existing methods employ only the values of the function itself, which may be suboptimal if the values of the gradient are also available for free.

\section{Gradient Cross and value function approximation}
\label{sec:gc}
As discussed in the previous sections, we want to recover the feedback map for the optimal control problem given the knowledge of the value function and its gradient in specific points.
In this section we develop a new algorithm in which this information about the value function will enrich the approximation via the FTT. More precisely, given certain sample points $\{x_i \}_{i=1}^N$ and a dataset $\{ \,V(x_i),\, \nabla V(x_i)\}_{i=1}^N$ computed by either Pontryagin or SDRE, we are interested in determining the coefficient tensors $\{H^{(1)},\ldots, H^{(d)}\}$ which characterize the FTT representation $\tilde{V}(x)$ introduced in \eqref{FTTd}. The regression problem can be formulated as
$$
\min_{H^{(1)}, \ldots, H^{(d)}} \sum_{i=1}^N | \tilde{V}(x_i) - V(x_i)|^2 + \lambda \Vert \nabla \tilde{V}(x_i) -\nabla V(x_i) \Vert^2,
$$
where $\lambda>0$ is a parameter which weights the contribution of the derivatives, and the gradient of the FTT can be computed considering univariate derivatives \eqref{grad_ftt}.
The resolution of the minimization problem will be addressed in the next sections, first presenting the bi-variate case, and then extending the algorithm to higher dimensions.

\subsection{Bidimensional Gradient Cross}

Given a bidimensional function $V(x_1,x_2)$, we will denote by $V_0(x_1,x_2)$ the function itself and by $V_1(x_1,x_2)$ and $V_2(x_1,x_2)$ its partial derivatives with respect to $x_1$ and $x_2$.
Let us fix two sets of collocation points $X_1 \in \mathbb{R}^{n_1}$ and $X_2 \in \mathbb{R}^{n_2}$ for each dimension. One may use the matrix $V_i(X_1,X_2)=[V_i(x_1^{(j)},x_2^{(k)})] \in \mathbb{R}^{n_1 \times n_2}$ for $i \in \{0,1,2\}$ to construct a discrete approximation of the function, but these evaluations are expensive, in particular in dimension larger than two.

For this reason we are interested in a TT representation of the form
\begin{equation}
V(x_1,x_2) \approx G^{(1)}(x_1) G^{(2)}(x_2)
\label{gradient_cross}
\end{equation}
with
$$
G^{(1)}(x_1) = \Phi_1(x_1) H^{(1)} , \qquad G^{(2)}(x_2) = H^{(2)}  \Phi^{\top}_2(x_2)
$$
where $H^{(1)} \in \mathbb{R}^{n_1 \times r}$, $H^{(2)} \in \mathbb{R}^{r \times n_2}$ and $\{\Phi_k(x)\}_{k=1,2}$ are prefixed basis functions.

More precisely, we are looking for two sets of indices $I_1$ and $I_2$ with cardinality $\#I_1=\#I_2=r$, and the corresponding interpolating approximations
$$
G^{(1)}(x_1) \hat G^{(2)}(x_2) \qquad \mbox{and} \qquad \hat G^{(1)}(x_1) G^{(2)}(x_2)
$$
such that $G^{(k)}(X_k(I_k))=\mathtt{I}_r$, where $\mathtt{I}_r$ is the identity matrix,
and $\hat G^{(1)}(x_1) = V(x_1, X_2(I_2))$, $\hat G^{(2)}(x_2) = V(X_1(I_1), x_2)$.
This kind of approximation can be obtained via an alternating direction procedure by solving a sequence of least squares problems.

Starting from an initial guess for $H^{(2)}$ and $I_2$, we want to solve the following problem in the $x_1$-direction
\begin{equation}
\min_{H^{(1)}} \sum_{i=0}^2\lambda_i \Vert V_i(X_1,X_2(I_2))- \tilde{V}^1_i \Vert^2,
\label{min_x}
\end{equation}
with
$$
\tilde{V}^1_i= \begin{cases} \Phi_1(X_1) H^{(1)} G^{(2)}(X_2(I_2)) & i=0, \\
\Phi'_1(X_1) H^{(1)} G^{(2)}(X_2(I_2)) & i=1, \\
\Phi_1(X_1) H^{(1)} \partial_{x_2} G^{(2)}(X_2(I_2))  & i=2.
\end{cases}
$$
In what follows we will fix $\lambda_0=1$ and $\lambda_1=\lambda_2=\lambda$ are the regularization parameters.

Let us consider for simplicity Lagrangian basis, i.e. $ \Phi_k(X_k)= \mathtt{I}_{n_k}  \in \mathbb{R}^{n_k \times n_k}$ for $k=1,2$. Moreover, we know by hypothesis that $G^{(2)}(X_2(I_2))=\mathtt{I}_r$.
Then, \eqref{min_x} is equivalent to the resolution of the following Lyapunov equation
\begin{equation}
\left(\mathtt{I}_{n_1} +\lambda \Phi'_1(X_1)^{\top}\Phi'_1(X_1)\right) H^{(1)}+\lambda H^{(1)} \tilde{G}_2 \tilde{G}_2^{\top} = V_0(X_1,X_2(I_2))+\lambda \Phi_1'(X_1)^{\top}  V_1(X_1,X_2(I_2))+ \lambda V_2(X_1,X_2(I_2)) \tilde{G}_2^{\top}  \,,
\label{lyap_x}
\end{equation}
where $\tilde{G}_{2}=\partial_{x_2} G^{(2)}(X_2(I_2)) = (\Phi'_2(X_2(I_2)) \Phi_2(X_2(I_2))^{\dagger})^\top$, where $\Phi_2(X_2(I_2))^{\dagger} = (H^{(2)})^\top$ is the Moore-Penrose pseudoinverse of $\Phi_2(X_2(I_2))$, arising from the condition $G^{(2)}(X_2(I_2))=\mathtt{I}_r$. 

Having solved equation \eqref{lyap_x}, we first execute the QR decomposition of the matrix $H^{(1)}=H R_1$ to improve the numerical stability and then we apply the maximum volume ($maxvol$) method \cite{gostz-maxvol-2010} to the matrix $H$. The $maxvol$ algorithm selects the most relevant indices $I_1$ such that the coefficient matrix $C:=H H[I_1,:]^{-1}$ satisfies $\max_{i,j}|C[i,j]| \le 1 + \delta$, where $\delta >0 $ is an arbitrary threshold. This procedure will provide an approximation of the maximum volume submatrix $H[I_*,:]$, $i.e.$ the submatrix with maximum determinant in modulus among all the possible $r \times r$ submatrices.
Taking $C$ as the  new $H^{(1)}$
ensures $G^{(1)}(X_1(I_1)) = \mathtt{I}_r$ for the next steps.

%Starting from an initial guess for $H^{(2)}_\alpha$ and $I_2$, we want to solve the following problem in the $x$-direction
%
%\begin{equation}
%\min_{S} \sum_{i=0}^2\lambda_i \Vert h_i(X_1,X_2(I_2))- \tilde{h}^2_i(X_1,X_2(I_2)) \Vert^2 ,
%\label{min_x}
%\end{equation}
%
%with
%$$
%\tilde{h}^2_i= \begin{cases} \Phi_1(X_1)H^{(1)} 
%\end{cases}
%$$
%$$
%\phi^i_1(x)= \begin{cases}  \partial_x \phi_1(x) & i=1,\\
%                \phi_1(x) & i=0,2,
%              \end{cases}  \quad 
%              \phi^i_2(y)= \begin{cases} \partial_y \phi_2(y) & i=2,\\
%                 \phi_2(y) & i=0,1.
%              \end{cases} 
%$$

Afterwards, we can pass to the least squares problem in the $y$-direction
\begin{equation}
\min_{H^{(2)}} \sum_{i=0}^2\lambda_i \Vert V_i(X_1(I_1),X_2)- \tilde{V}^2_i \Vert^2,
\label{min_y}
\end{equation}
with
$$
\tilde{V}^2_i= \begin{cases} G^{(1)}(X_1(I_1)) H^{(2)}  \Phi_2(X_2)^{\top}  & i=0, \\
\partial_{x_1} G^{(1)}(X_1(I_1)) H^{(2)}  \Phi_2(X_2)^{\top}  & i=1, \\
G^{(1)}(X_1(I_1)) H^{(2)}  \Phi'_2(X_2)^{\top} & i=2.
\end{cases}
$$
After the first step we have $G^{(1)}(X_1(I_1))=\mathtt{I}_r$ by construction. Then, \eqref{min_y} corresponds to the resolution of a Lyapunov equation
\begin{equation}
\lambda \tilde{G}_1^{\top} \tilde{G}_1 H^{(2)}+  
H^{(2)}\left(I_{n_2} +\lambda \Phi'_2(X_2)^{\top}\Phi'_2(X_2)\right) = V_0(X_1(I_1),X_2)+\lambda \tilde{G}_1^{\top}  V_1(X_1(I_1),X_2)+ \lambda V_2(X_1(I_1),X_2) \Phi'_2(X_2) \,,
\label{lyap_y}
\end{equation}
where $\tilde{G}_1=\partial_{x_1} G^{(1)}(X_1(I_1))$.
The remaining procedure is identical: we solve \eqref{lyap_y}, compute the QR decomposition of the solution and apply the $maxvol$ method, obtaining $I_2$ and $H^{(2)}$.
The strategy is repeated until either we converge according to residual criteria or we reach a maximum number of iterations.
The method is sketched in Algorithm \ref{alg_1}.

\begin{algorithm}[H]
%\captionsetup[algorithm]{name=TSA}
\caption{Bidimensional TT Gradient Cross with Lagrangian bases}
\begin{algorithmic}[1]
\State Choose an initial $I_2$ and $H^{(2)}$, a tolerance $tol$ and a maximum number of iteration $it_{max}$
\While{$res > tol$ and $it \le it_{max}$}
\State{Solve equation \eqref{lyap_x} obtaining $H^{(1)}$}
\State{Compute the QR decomposition $H^{(1)}=HR_1$}
\State{$[I_1,C]=maxvol(H)$}, replace $H^{(1)} = C.$\label{alg_1_maxvol_l}
\State{Solve equation \eqref{lyap_y} obtaining $H^{(2)}$}
\State{Compute the QR decomposition $(H^{(2)})^{\top}=H R_1$}\label{alg_1_qr_r}
\State{$[I_2,C]=maxvol(H)$}, replace $H^{(2)} = C^\top.$\label{alg_1_maxvol_r}
\State{Update $res$}\label{alg_1_res}
\State{$it=it+1$}
\EndWhile
\end{algorithmic}
\label{alg_1}
\end{algorithm}

For the initial guess $H^{(2)}$ one may consider a normally distributed pseudorandom matrix $H^{(2)} \in \mathbb{R}^{r \times n_2}$ and apply steps~\ref{alg_1_qr_r} and \ref{alg_1_maxvol_r} of Algorithm~\ref{alg_1} to obtain the initial set $I_2$.

With non-Lagrangian bases (i.e. if $\Phi_k(X_k) \neq \mathtt{I}_{n_k}$),
we simply need to amend lines~\ref{alg_1_maxvol_l} and \ref{alg_1_maxvol_r} of Algorithm~\ref{alg_1} to
$[I_1,\Phi_1(X_1) H^{(1)}]=maxvol(\Phi_1(X_1) H)$
and
$[I_2,\Phi_2(X_2) (H^{(2)})^{\top}]=maxvol(\Phi_2(X_2) H)$,
respectively.
This ensures that we select optimal grid points, not optimal coefficients.

The most reliable residual criterion is to compute the mean square approximation error on some validation set (e.g. random points), and to compare it to a chosen threshold.
However, a large validation set may inflate the computing time significantly, while a small set may underestimate the error.
In the course of existing alternating algorithms \cite{dkos-eigb-2014} it was found that it is sufficient to compare consecutive iterations.
Therefore, we proceed with the following definition in line~\ref{alg_1_res} of Algorithm~\ref{alg_1}:
$$
res = \max\left\{\frac{\|H^{(1)}_{it} - H^{(1)}_{it-1}\|_F}{\|H^{(1)}_{it}\|_F},~\frac{\|H^{(2)}_{it} - H^{(2)}_{it-1}\|_F}{\|H^{(2)}_{it}\|_F}\right\},
$$
where $it$ is the iteration number.

Once the value function is approximated by Algorithm~\ref{alg_1}, we can compute the optimal control and optimal trajectory starting from a given initial point $x^0 \in \mathbb{R}^2$.
The formula for the synthesis of the optimal control is given by
$$
u(x)=-\frac{1}{2}R^{-1} B(x)^{\top} \nabla V(x),
$$
where $x=(x_1,x_2) \in \mathbb{R}^2$. 
The computation of the gradient in this case is simply given by considering the derivatives in \eqref{gradient_cross}
$$
\partial_{x_1} V(x)=\Phi'_1(x_1) H^{(1)} G_2(x_2) ,
$$
$$
\partial_{x_2} V(x)=   G_1(x_1) H^{(2)} \Phi_2'^{\top}(x_2).
$$

\subsection{Multidimensional Gradient Cross}
\label{sec:multiTT}
Now we are going to generalize the result obtained in the previous section to an arbitrary dimension $d$. 
% First of all, let us fix for each variable $x_k$, a set of $n_k$ basis functions $\{\Phi^{(1)}_k(x_k),\dots , \Phi^{(n_k)}_k(x_k) \}$ and a set of collocation points $X_k=\{x_k^{(i)}\}_{i=1}^{n_k}$.
%Now, given a multivariate function $h:X = \varprod_{k=1}^d X_k  \rightarrow \mathbb{R}$, we are interested in the following approximation
%\begin{equation}
%h(x) \approx \tilde{h}(x) = \sum_{\alpha_0=1}^{r_0} \sum_{\alpha_1=1}^{r_1} \, \dotsi \,\sum_{\alpha_d=1}^{r_d} G^{(1)}_{(\alpha_0,\alpha_1)}(x_1) \, \dotsi \, G^{(k)}_{(\alpha_{k-1},\alpha_k)}(x_k) \,\dotsi\, G^{(d)}_{(\alpha_{d-1},\alpha_d)}(x_d).
%\label{FTTd}
%\end{equation}
%The term $G^{(k)}_{(\alpha_{k-1},\alpha_k)}(x_k)$ is called $k-th$ TT core and it is a linear combination of the prefixed $n_k$ basis:
%$$
%G^{(k)}_{(\alpha_{k-1},\alpha_k)}(x_k)= \sum_{i=1}^{n_k} \Phi^{(i)}_k(x_k) H^{(k)}_{(\alpha_{k-1},i,\alpha_k)}=\Phi_k(x_k)\cdot H^{(k)}_{(\alpha_{k-1},\alpha_k)}\;,
%$$
%where $ H^{(k)}_{(\alpha_{k-1},i,\alpha_k)} \in \mathbb{R}^{r_{k-1}\times n_k \times r_k}$ is a three-dimensional tensor.
%To ease the notation, we are going to write \eqref{FTTd} in matrix form as follows
%$$
%\tilde{h}(x) =  G^{(1)}(x_1)\, \dotsi \,G^{(k)}(x_k)\, \dotsi \, G^{(d)}(x_d).
%$$ 
The FTT representation in this case reads
$$
\tilde{V}(x) =  G^{(1)}(x_1)\, \dotsi \,G^{(k)}(x_k)\, \dotsi \, G^{(d)}(x_d).
$$ 
We will use the alternating strategy in this case too.
For this reason it is convenient to group all the terms before and after the $k$-th TT core, obtaining a more compact formula
$$
\tilde{V}(x) =  G^{(<k)}(x_{<k})\cdot G^{(k)}(x_k)\cdot  G^{(>k)}(x_{>k}),
$$
%$$
%\tilde{h}(x^{(\alpha)}_{<k},x_k^{(i)},x_{>k}^{(\beta)}) =  G^{(<k)}(x^{(\alpha)}_{<k})G^{(k)}(x_k)  G^{(>k)}(x^{\beta}_{>k}),
%$$
where
\begin{align}
G^{(<k)}(x_{<k})&=G^{(1)}(x_1)\, \dotsi \,G^{(k-1)}(x_{k-1}), & k &\ge 2, \\
G^{(>k)}(x_{>k})&=G^{(k+1)}(x_{k+1})\, \dotsi \, G^{(d)}(x_d), & k &\le d-1.
\end{align}

We are again interested in finding interpolation sets $\overline{X}_{<k} \subset X_1 \times \cdots \times X_{k-1}$ and $\overline{X}_{>k} \subset X_{k+1} \times \cdots \times X_{d}$ with $r_{k-1}$ and $r_k$ points, respectively.
Let us suppose that in the $k$-th step the sets $\overline{X}_{<k}$ and $\overline{X}_{>k}$ are given.
Combining the previous expressions, we can write
$$
\vv{V}^k:=\mathrm{vec}\left(V(\overline{X}_{<k},X_k,\overline{X}_{>k})\right)\approx \tilde{V}^k:= \left( G^{(<k)}\left(\overline{X}_{<k}\right) \otimes \Phi_k\left(X_k\right)  \otimes G^{(>k)}\left(\overline{X}_{>k}\right) \right) \cdot \mathrm{vec}(H^{(k)}),
$$
where $\mathrm{vec}(\cdot)$ stretches a tensor into a vector with the same order of elements,
and $\otimes$ is the Kronecker product of matrices.
Similarly to the bidimensional case, we are going to denote the function values by $\vv{V}^k_0$ and the derivative values with respect to the $i$-th component by $\vv{V}^k_i$.
Our aim is to solve the following least squares problem
\begin{equation}
\min_{H^{(k)}} \sum_{i=0}^{d} \lambda_i \Vert \vv{V}^k_i -\tilde{V}^k_i \Vert^2,
\label{LSd}
\end{equation}
where
$$
\tilde{V}^k_i= \begin{cases} \left( G^{(<k)}\left(\overline{X}_{<k}\right) \otimes \Phi_k\left(X_k\right)  \otimes G^{(>k)}\left(\overline{X}_{>k}\right) \right) \cdot \mathrm{vec}(H^{(k)})  & i=0 \\
\left( \partial_i  G^{(<k)}\left(\overline{X}_{<k}\right) \otimes \Phi_k\left(X_k\right)  \otimes G^{(>k)}\left(\overline{X}_{>k}\right)  \right) \cdot \mathrm{vec}(H^{(k)})  & i=1,\ldots,k-1 \\
\left( G^{(<k)}\left(\overline{X}_{<k}\right) \otimes \Phi'_k\left(X_k\right)  \otimes  G^{(>k)}\left(\overline{X}_{>k}\right)  \right) \cdot \mathrm{vec}(H^{(k)})  & i=k \\
 \left(  G^{(<k)}\left(\overline{X}_{<k}\right) \otimes \Phi_k\left(X_k\right)  \otimes \partial_i G^{(>k)}\left(\overline{X}_{>k}\right)  \right) \cdot \mathrm{vec}(H^{(k)}) & i=k+1,\ldots,d.
\end{cases}
$$
\begin{proposition}
The least square problem \eqref{LSd} can be solved as a three-dimensional Sylvester equation.
\end{proposition}
\begin{proof}
The proof is available in Appendix~\ref{appendixA1}.
\end{proof}
%This least squares problem can be solved as a three-dimensional Sylvester equation as shown in Appendix~\ref{appendixA1}.
Having obtained the solution $H^{(k)},$
we can
consider the unfolding matrix $H^{(k)}_L = \left[H^{(k)}_{(\overline{\alpha_{k-1}, i},\alpha_k)}\right] \in \mathbb{R}^{r_{k-1}n_k \times r_k}$
and compute its QR factorization $H^{(k)}_L = \tilde H^{(k)}_L R^{(k)}$,
then assemble an unfolded TT core
$$
\tilde G_L^{(k)} = \left(\mathtt{I}_{r_{k-1}} \otimes \Phi_k(X_k) \right) \tilde H_L^{(k)} \quad \Leftrightarrow \quad \tilde G^{(k)}(x_k)= \sum_{i=1}^{n_k} \Phi_k^{(i)}(x_k) \tilde H^{(k)}_{(i)}.
$$
%then reshape the orthogonal matrix $\tilde H^{(k)}_L$ back into a tensor $\tilde H^{(k)} \in \mathbb{R}^{r_{k-1} \times n_k \times r_k}$ and
%assemble the three-dimensional core
%$$
%\tilde G^{(k)}(x_k)= \sum_{i=1}^{n_k} \Phi_k^{(i)}(x_k) \tilde H^{(k)}_{(i)} \in \mathbb{R}^{r_{k-1} \times r_k},
%$$
%and the unfolding matrix $G^{(k)}_L= G^{(k)}(\overline{\alpha_{k-1} X_k},\alpha_k) \in \mathbb{R}^{r_{k-1}n_k \times r_k}$.
Now we can use the $maxvol$ method on $\tilde G^{(k)}_L$ to find the set $I_k$,
which is a subset of $[\alpha_{k-1}]_{\alpha_{k-1}=1}^{r_{k-1}} \times [i]_{i=1}^{n_k}$.
We can split $I_k$ into corresponding components
\begin{equation}\label{eq:Ik-split}
I_k^{\alpha} = \{\alpha_{k-1}: \overline{\alpha_{k-1},i} \in I_k\}, \quad \mbox{and} \quad I_k^{x} = \{i: \overline{\alpha_{k-1},i} \in I_k\}.
\end{equation}
In turn, those enumerate elements in $\overline{X}_{<k}$ and $X_k$.
Therefore, we can define the new interpolation set
\begin{equation}\label{eq:X_k+1}
\overline{X}_{<k+1} := \overline{X}_{<k}(I_k^{\alpha}) \times X_k(I_k^x).
\end{equation}
that allows us to continue the iteration.
Finally, the new TT core tensor is recovered from the interpolating unfolding matrix $H^{(k)}_L:=\tilde H^{(k)}_L \left(\tilde G^{(k)}_L[I_k,:]\right)^{-1}$.

Passing on to the $(k+1)$-th step,
we need to compute $G^{(<k+1)}\left(\overline{X}_{<k+1}\right)$.
Computing the corresponding evaluations of all $k$ cores constituting $G^{(<k+1)}$ will result in an $\mathcal{O}(d^2)$ complexity of the entire algorithm.
However, since we can assume that $G^{(<k)}\left(\overline{X}_{<k}\right)$ was available in the current step, we can obtain $G^{(<k+1)}\left(\overline{X}_{<k+1}\right)$ with a cost independent of $k$ (and $d$).
These computations are shown in Appendix~\ref{appendixA2}.
%Indeed, for each of the $r_k$ elements in $I_k^{\alpha}$ and $I_k^x$ we compute the matrix products
%\begin{align}\label{eq:G_k+1}
%G^{(<k+1)}\left(\overline{X}_{<k+1}(\alpha_k)\right) &= G^{(<k)}\left(\overline{X}_{<k}(I_k^{\alpha}(\alpha_k))\right) G^{(k)}\left(X_{k}(I_k^x(\alpha_k))\right), \\\nonumber
%\partial_i G^{(<k+1)}\left(\overline{X}_{<k+1}(\alpha_k)\right) & =  \partial_i G^{(<k)}\left(\overline{X}_{<k}(I_k^{\alpha}(\alpha_k))\right)  G^{(k)}\left(X_{k}(I_k^x(\alpha_k))\right), &  i& =1,\ldots,k-1, \\\nonumber
%\partial_k G^{(<k+1)}\left(\overline{X}_{<k+1}(\alpha_k)\right) & = G^{(<k)}\left(\overline{X}_{<k}(I_k^{\alpha}(\alpha_k))\right)  \partial_k  G^{(k)}\left(X_{k}(I_k^x(\alpha_k))\right), & \alpha_k & =1,\ldots,r_k.
%\end{align}

We proceed in the same fashion until either the discrepancy between the consecutive iterations is below the stopping tolerance, or a maximum number of iteration has been reached.

The method needs a little modification to adapt the TT ranks to a given error threshold.
Firstly, if the ranks are overestimated, we can reduce them by computing the singular value decomposition (SVD) instead of the QR decomposition of $H^{(k)}_L$,
and truncate the former to ensure that the sum of squares of the truncated singular values is below the desired threshold.
Secondly, if the ranks are underestimated, we can expand $H^{(k)}_L$ with some extra $\rho_k$ columns $Z^{(k)}_L$, thereby increasing the TT rank $r_k$ to $r_k+\rho_k$.
It was shown~\cite{ds-amen-2014} that $Z^{(k)}$ can be taken as TT cores of a TT approximation of the error $z(x) := V(x) - \tilde V(x)$.
This interplay of the reduction and expansion of the TT ranks will eventually stabilize near optimal ranks for the given error.
The entire procedure is summarized in Algorithm~\ref{alg:gradient_cross}.
The expansion cores $Z^{(k)}$ are obtained by running an independent instance of the same algorithm computing a TT approximation $\tilde z(x) \approx z(x)$ with fixed TT ranks $\rho_1=\ldots=\rho_{d-1}=\rho$ instead of lines \ref{alg_gc_svd} and \ref{alg_gc_enrich}.

\begin{algorithm}[htb]
\caption{Gradient TT Cross}
\begin{algorithmic}[1]
\State Choose initial TT cores $H^{(k)}$, point sets $\overline{X}_{>k}$, a tolerance $tol$ and a maximum number of iterations $it_{max}$.
\While{$res > tol$ and $it \le it_{max}$}
\For{$k=1,\ldots,d$}
\State{Find $H^{(k)}$ as the minimizer in \eqref{LSd} by solving \eqref{normal_eq}.}
\State{(Optionally) Reduce the TT rank $r_k$ via truncated SVD of $H^{(k)}_L$ to error threshold $tol$}\label{alg_gc_svd}
\State{(Optionally) Increase $r_k$ by expanding $H^{(k)}_L:=\left[H^{(k)}_L,~Z^{(k)}_L\right]$ with TT core of error $Z^{(k)}_L$}\label{alg_gc_enrich}
\State{Compute the QR decomposition $H^{(k)}_L=\tilde H^{(k)}_L R^{(k)}$}
\State{$\left[I_k,~(\mathtt{I}_{r_{k-1}} \otimes \Phi_k(X_k)) H^{(k)}_L\right]=maxvol\left((\mathtt{I}_{r_{k-1}} \otimes \Phi_k(X_k)) \tilde H^{(k)}_L\right)$}\label{alg_gc_maxvol}
\State{Compute the next point set $\overline{X}_{<k+1}$ as shown in Eq.~\eqref{eq:Ik-split},~\eqref{eq:X_k+1}}
\State{Compute the next sampled cores as shown in Eq.~\eqref{eq:G_k+1}}
\State{Update $res$}\label{alg_gc_res}
\EndFor
\State{$it=it+1$}
\EndWhile
\end{algorithmic}
\label{alg:gradient_cross}
\end{algorithm}

\subsection{Two Boxes approach}

Our scope is to solve the HJB equation \eqref{HJB} in a computational domain $\Omega$ which is usually a hypercube $[-a,a]^d$ containing all initial conditions of interest and their subsequent optimal trajectories. We discretize the domain separately along each dimension. For the high dimensional numerical tests we consider Gauss-Legendre nodes and the corresponding Legendre polynomials on the nodes. This choice allow us to obtain accurate solution in the vicinity of the boundary, but it may be less accurate closer to the origin, where the system stabilizes.
For this reason we introduce an additional step: the Two Boxes (TB) algorithm.
First of all, we solve the optimal control problem on the whole domain using the TT Gradient Cross, constructing an approximation of the value function $V$ on $\Omega$. Afterwards, we construct the optimal trajectory $\tilde y^0(t)$ and the optimal control $\tilde u^0(t)$ starting from the origin, $\tilde y^0(0) = 0$. The exact path and control are zero constant functions since the origin is an equilibrium of the dynamics, but the approximation may escape from the origin due to approximation errors. In this case we consider the maximum value reached by the dynamics until a final time $T$. This maximum will be denoted as $\tilde y^0_{\max}:= \max_t \max_i |\tilde y_i^0(t)|$, and we will set $a_{TB}=2 \tilde y^0_{\max}$.
Afterwards, we construct a new value function $V_{TB}$ on the sub-domain $[-a_{TB},a_{TB}]^d$. Since the new domain is closer to the origin and smaller than the entire domain, we expect a tensor with smaller TT ranks and evaluations needed in the gradient cross. We will use the information provided by both value functions, defining the optimal feedback map as
$$
u^*(x)= \left\{ \begin{array}{l}
F(\nabla V_{TB}(x)), \quad \Vert x \Vert_\infty \le  a_{TB} ,\\
F(\nabla V(x)),\quad otherwise ,
\end{array} \right.
$$
with $F(g(x))=-\frac{1}{2}R^{-1}B(x)^{\top}g(x)$.

If $\tilde y^0_{\max}$ is small enough, the value function $V_{TB}$ would be close to the solution of the Linear Quadratic Regulator (LQR) problem. In the LQR setting the value function reads $V_{LQR}=x^{\top} \Pi x$, where $\Pi$ is the solution of the Riccati equation
$$
A^{\top}(0) \Pi + \Pi A(0) - \Pi B(0) R^{-1} B(0)^{\top} \Pi+Q=0 \,.
$$
In this case we can construct the optimal control as
$$
u^*(x)= \left\{ \begin{array}{l}
-R^{-1}B(0)^{\top} \Pi x, \quad \Vert x \Vert_{\infty} \le  a_{TB} ,\\
-\frac{1}{2}R^{-1}B(x)^{\top} \nabla V(x),\quad otherwise .
\end{array} \right.
$$
We notice that the second choice provides a faster procedure, since it implies just one resolution of a Riccati equation.
%In the cases studied in the section of the numerical tests we will notice that the second choice obtains better results in terms of error and it is much faster, since it implies just one computation of a Riccati equation.

\section{Numerical Tests}
\label{sec:nt}In this section we assess the proposed methodology through different numerical tests. First, we investigate the effect of adding gradient information in the regression in a series of closed-form high-dimensional functions with noisy evaluations.
The second numerical test deals with a two dimensional optimal control problem in which the exact value function is known. We test the efficiency of the method under noise and the effect of introducing control constraints. In the third example we study the three-dimensional Lorenz system. We study the performance of the algorithm reducing the control energy penalty in the cost functional. The last test deals with the Cucker-Smale model, where first we compare the PMP and SDRE approaches for data generation. We study the effect of varying the parameter $\lambda$ and the selection of an optimal parameter. In the last part of the section, a comparison with a Neural Network approach is presented together with the application of the Two Boxes approach.
The numerical simulations reported in this paper are performed on a Dell XPS 13 with Intel Core i7, 2.8GHz and 16GB RAM. The codes are written in Matlab R2021a.

Let us introduce some notations useful for the next sections. We denote by $J_{T}$ the total discrete cost functional computed by applying directly SDRE to approximate the optimal control problem up to a fixed final time $T$. The total discrete cost functional computed via TT Gradient Cross up to $T$ will be denoted by $\tilde{J}_{T}$.
Similarly, $\tilde{y}^*(t)$ denotes the discrete trajectory at time $t$ controlled with TT, and $\tilde{u}^*(t)$ denotes the corresponding discrete optimal control at time $t$.

We define the errors in the computation of the cost function and optimal control as
$$
err_J:=|J_{T}-\tilde{J}_{T}|, \qquad
err_u:=\sqrt{\sum_{i=0}^{n_t-1} (t_{i+1}-t_i) |u(t_i)-\tilde{u}(t_i)|^2},
$$
respectively, 
where $n_t$ is the number of time steps $t_i$ produced by the RK4 ODE solver,
 and the maximum absolute optimal state value at the final time $T$ as
$$
\tilde{y}_{\max}(T):=\max_i |\tilde{y}_i^*(T)|.
$$

\subsection{High-dimensional function approximation}

In this first numerical experiment we test the gradient cross algorithm for the approximation of high-dimensional functions. We will study the behaviour of the algorithm in presence of different noise levels.
We are going to consider the following two functions in dimension $d$:
\begin{itemize}
\item[a)] $f(x)=\exp(-\sum_{i=1}^d x_i/(2d) )$,  $x\in [-1,1]^d$,
\item[b)] $f(x)=\exp(-\prod_{i=1}^d x_i )$, $x\in [-1,1]^d$.
\end{itemize}

We fix the dimension $d=100$, the stopping error threshold for the gradient cross $tol=10^{-4}$ and we discretize the interval $[-1,1]$ with 33 Legendre-Gauss nodes for each direction.
The noise is introduced by adding independent identically distributed normal random numbers with mean $0$ and standard deviation $\sigma$ to the values of the function $f(x)$ and all components of the gradient $\partial_i f(x)$.
The error is computed with respect to the adaptive TT-Cross \cite{sav-qott-2014} with tolerance $10^{-12}$ and in absence of noise.

It is easy to see that the function $(a)$ has an exact rank-1 TT decomposition, so we will run the TT-Cross with a fixed rank 1.
In Figure~\ref{hig_dim} we show a comparison in terms of the mean approximation error for different $\lambda$. The noise magnitude $\sigma$ varies in the set $\{0\} \cup \{10^{-k}, k=1,\ldots,6\}$. In absence of noise, the gradient cross with $\lambda=0$ performs with high precision. Increasing the noise, the higher is $\lambda$, the better is the approximation.

Function $(b)$ is a function of rank higher than 1, and it is already possible to notice a different behaviour. For every noise amplitude it is possible to find a $\lambda \neq 0$ which obtains a better result compared to the cross approximation without gradient knowledge.
In particular, the gradient cross gives a meaningful approximation with an error of $0.1$ even with the largest noise magnitude $\sigma=0.1$, in which case the error of the standard TT-Cross (corresponding to $\lambda=0$) is larger than $1$.

	\begin{figure}[htbp]	
\centering
	\includegraphics[scale=0.5]{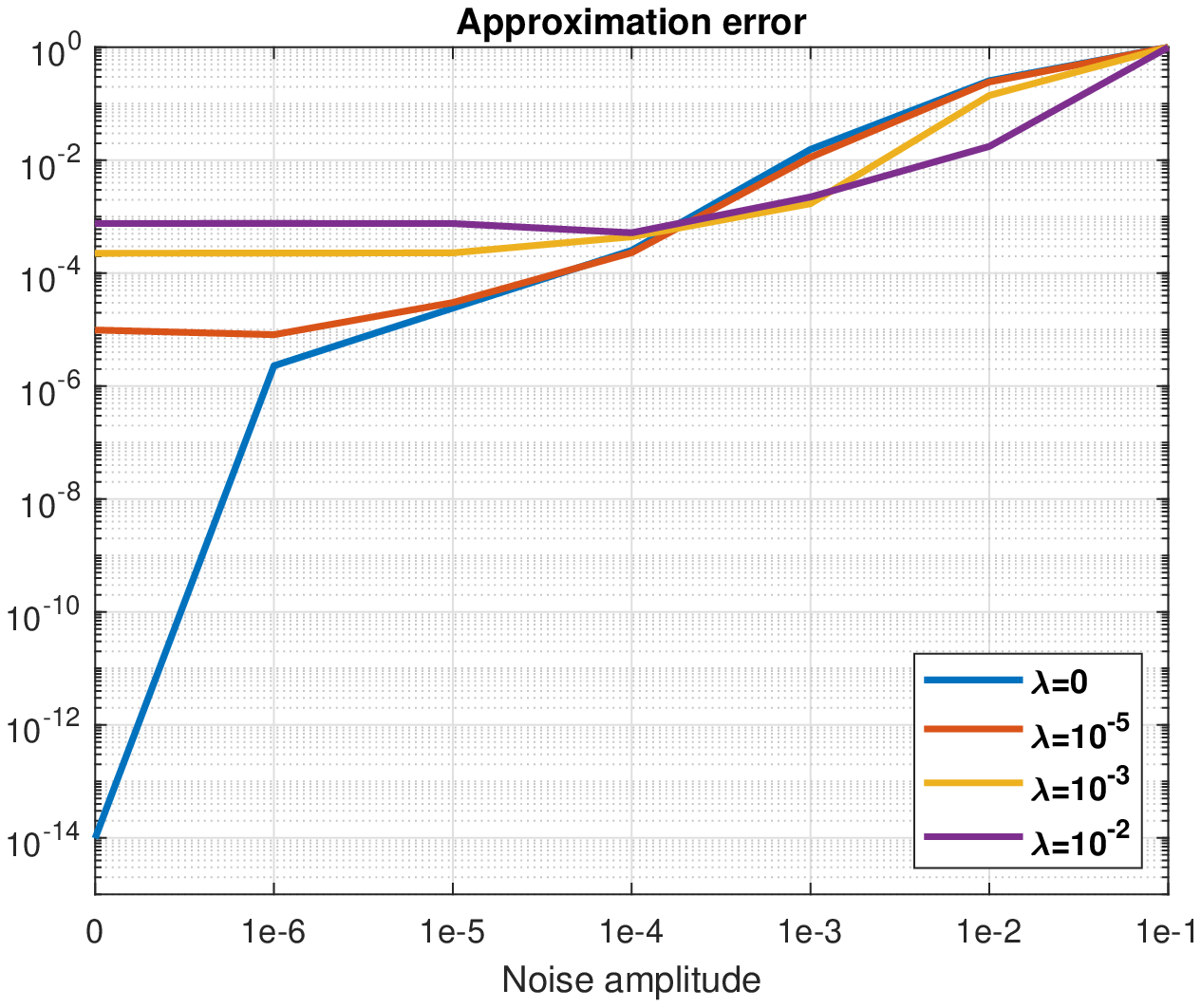}	
	\includegraphics[scale=0.5]{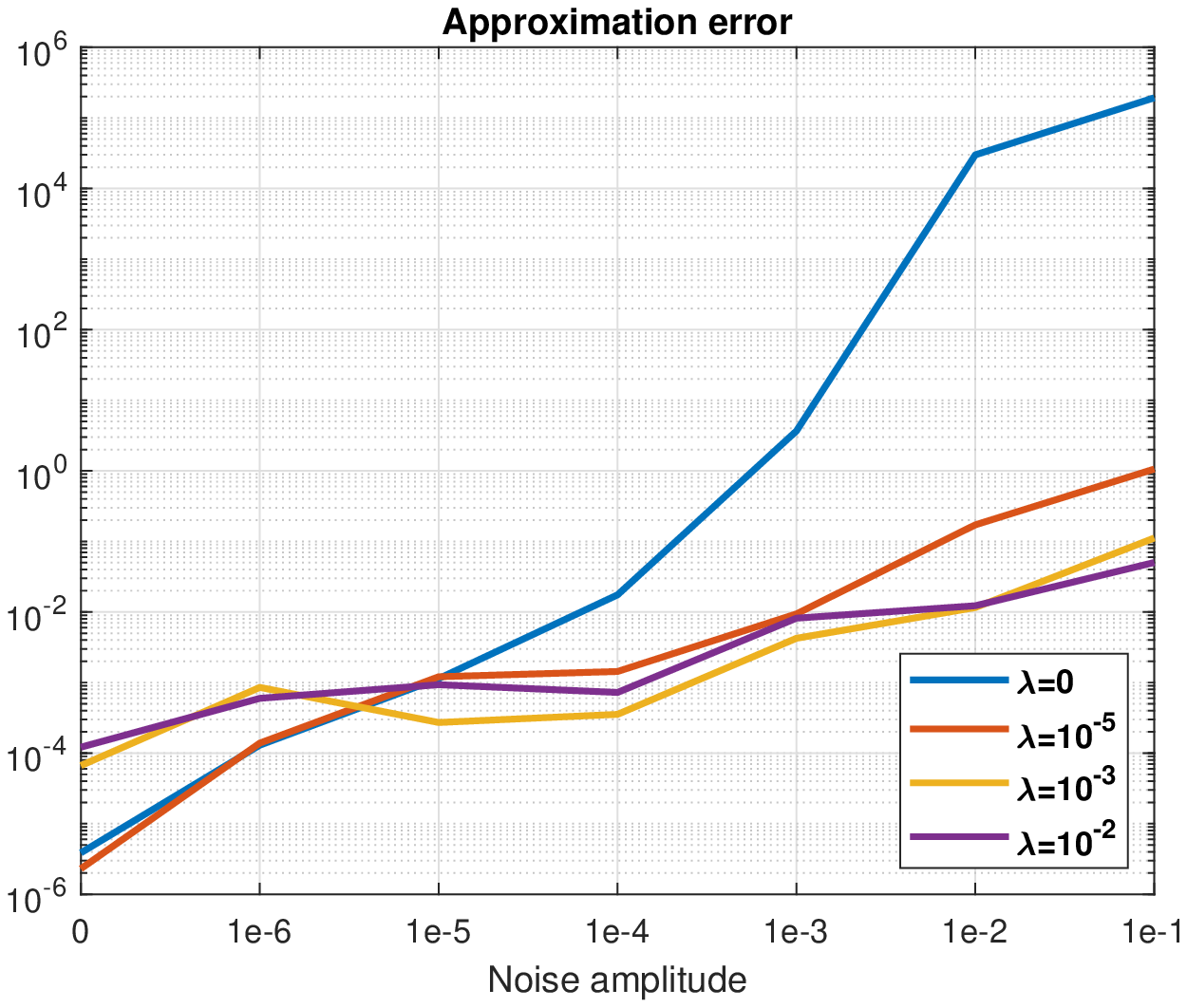}	
\caption{Mean approximation error for function $(a)$ (left) and function $(b)$ (right) for different $\lambda$ and noise amplitudes $\sigma$. The gradient cross (with $\lambda>0$) is much more accurate than a gradient-free method ($\lambda=0$) if a noisy approximate TT approximation is sought.}
\label{hig_dim}
\end{figure}

\subsection{2D Dynamics with Exact Solution}

The second numerical test deals with an example with exact solution.
Given the following state dynamics
\begin{equation}
\begin{bmatrix} \dot{x}_1 \\ \dot{x}_2  \end{bmatrix} = \begin{bmatrix} 0 & 1 \\ x_1^2 & 0  \end{bmatrix} \begin{bmatrix} x_1 \\ x_2  \end{bmatrix} + \begin{bmatrix} 0 \\ 1  \end{bmatrix} u
\label{dyn2D}
\end{equation}
and the associated cost functional
\begin{equation}
J= \frac{1}{2}\int_0^\infty \left( \Vert x(s) \Vert^2 + |u(s)|^2 \right) \, ds,
\label{J2D}
\end{equation}
the solution of the corresponding SDRE is
\begin{equation}
\Pi(x)= \begin{bmatrix} \frac{\sqrt{x_1^4 + 1}\sqrt{2\sqrt{x_1^4 + 1} + 2x_1^2 + 1}}{2} & \frac{\sqrt{ x_1^4 + 1} + x_1^2}{2} \\ \frac{\sqrt{ x_1^4 + 1} + x_1^2}{2} & \frac{\sqrt{2\sqrt{x_1^4 + 1} + 2x_1^2 + 1}}{2} \end{bmatrix}.
\label{SDREsol}
\end{equation}
We apply the TT gradient cross described in the previous section and we test it under the effect of noises of different amplitude $\sigma$.
We discretize the interval $[-1,1]$ using 14 Lagrangian basis functions and we fix the stopping error threshold for the gradient cross $tol=10^{-4}$.
The collection of the data for the value function and its gradient is performed via the resolution of SDREs. We will consider the case without the knowledge of the gradient ($i.e.$ $\lambda=0$) and the case with $\lambda=10^{-4}$.  The mean errors in Table \ref{Table_err} are computed on a sample of 100 random initial conditions. Since the control is computed taking into account the gradient of the value function, the sum of the considered errors provide an $H^1$ error estimate. It is possible to notice that in all cases the introduction of the gradient information yields a better approximation.
In Figure \ref{Fig_test2D} we show the optimal trajectories and the optimal control computed starting from the initial condition $x_0=(1,-1)$. The right panel of the figure shows the visual coincidence of the three solutions without the presence of noise, which was intuitable by the first row of Table \ref{Table_err}. The left panel shows the comparison of the solutions under a noise amplitude $\sigma=10^{-2}$. The choice $\lambda=0$ and $\lambda=10^{-4}$ cannot retrieve the starting behaviour of the exact control signal, which is zero at the initial time, while fixing $\lambda=1$ we can observe a better match.
%\begin{table}[htbp]
%						
%			
%		\centering
%		\begin{tabular}{c|cc|cc}
%					\toprule
%					& \multicolumn{2}{|c|}{\bf $err_J$}   &  \multicolumn{2}{|c}{\bf $err_u$} \\ $\sigma$ & $\lambda=0$ & $\lambda=10^{-4}$ &$\lambda=0$ &$\lambda=10^{-4}$  \\
%					\midrule
%				0 &	5.0e-8  &  4.6e-8 &   7.0e-7 &    5.3e-7  \\
%					$10^{-4}$ &	9.5e-6  &  5.2e-6 &    8.7e-4 &    7.3e-4  \\
%					$10^{-3}$ &	 1.0e-4 &    6.8e-5 &    1.1e-2 &   8.8e-3  \\
%					$10^{-2}$ &	 1.2e-3 &    8.2e-4 &     1.7e-2 & 1.2e-2  \\
%					$10^{-1}$ &	 2.6e-2 &   1.0e-2 &     0.16 & 8.8e-2  \\
%
%
%			\bottomrule		
%		\end{tabular}	
%		
%							\caption{Errors in the cost functional of the 2D model and in the control for different amplitudes of noise. \label{Table_err}}
%	\end{table}

\begin{table}[htbp]

		\centering
		\begin{tabular}{c|ccc|ccc}
					\toprule
					& \multicolumn{3}{|c|}{\bf $err_J$}   &  \multicolumn{3}{|c}{\bf $err_u$} \\ $\sigma$ & $\lambda=0$ & $\lambda=10^{-4}$& $\lambda=1$ &$\lambda=0$ &$\lambda=10^{-4}$ & $\lambda=1$  \\
					\midrule
				0 &	 8.6-9 &  9.8e-9 & 2.6e-8 & 5.0e-7 & 4.4e-7 & 1.1e-6  \\
					$10^{-4}$ &	9.1e-6 & 6.2e-6 & 3.4e-6 & 3.5e-4& 1.1e-4 & 6.4e-5 \\
					$10^{-3}$ &	6.1e-5 &  6.2e-5 & 3.5e-5 & 2.6e-3 & 1.7e-3 & 8.7e-4 \\
					$10^{-2}$ &	6.8e-4 & 6.8e-4 & 3.2e-4 & 1.2e-2 &  1.1e-2 & 5.5e-3  \\
					$10^{-1}$ &	 6.6e-2 &   2.0e-2   & 6.2e-3  &  0.16 & 0.11 &   6.0e-2 \\

			\bottomrule		
		\end{tabular}	
		
							\caption{Mean errors in the cost functional of the 2D model and in the control for different amplitudes of noise. \label{Table_err}}
	\end{table}
	
%	\begin{figure}[htbp]	
%\centering
%	\includegraphics[scale=0.5]{pics/2DExact/traj_control_noise0.eps}	
%	\includegraphics[scale=0.5]{pics/2DExact/traj_control_noise1_e2.eps}	
%\caption{Optimal trajectory and control without noise (left) and with noise amplitude $\sigma=10^{-2}$ (right) starting from $x_0=(1,-1)$.}
%\label{Fig_test2D}
%\end{figure}

	\begin{figure}[htbp]	
\centering
	\includegraphics[scale=0.5]{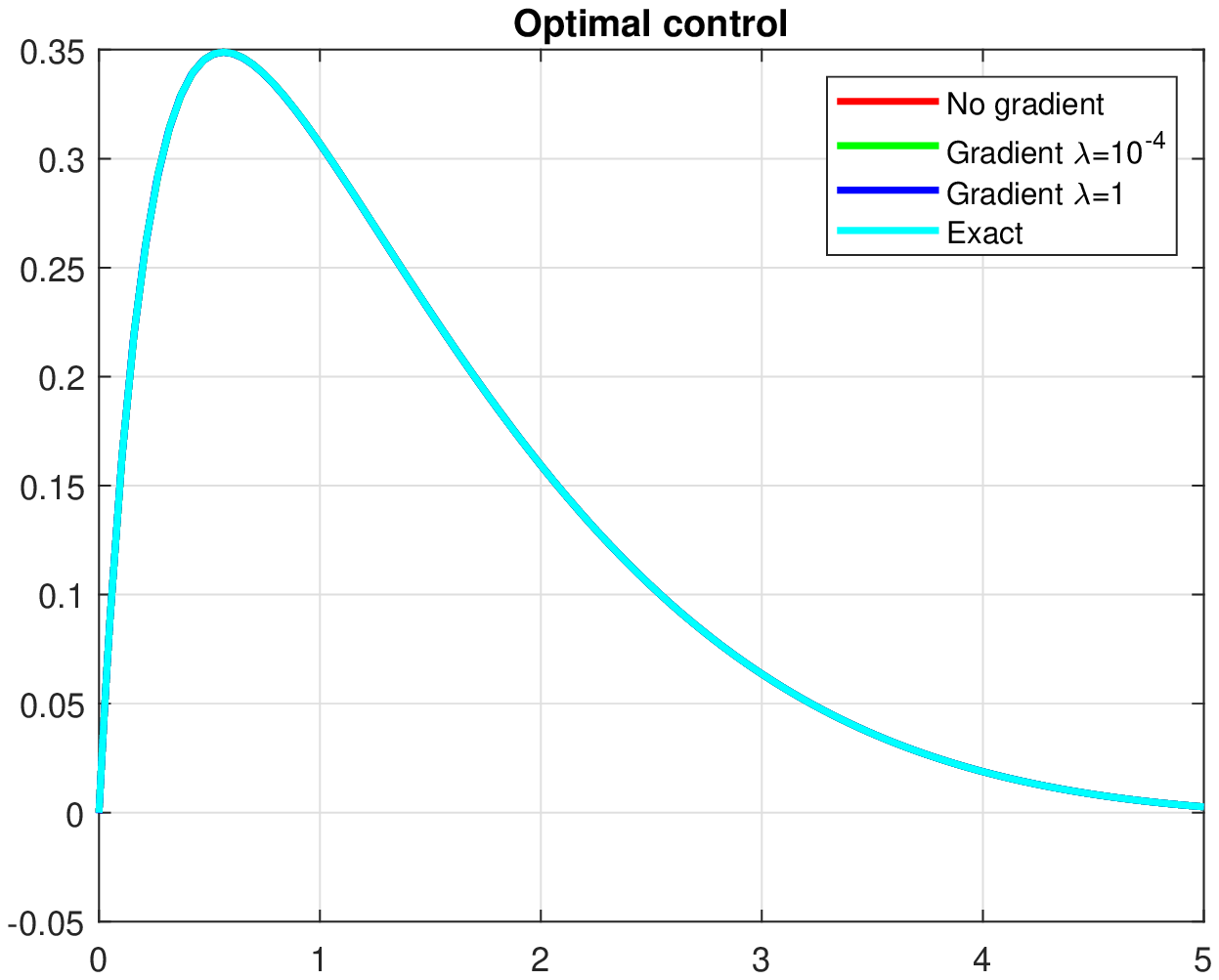}	
	\includegraphics[scale=0.5]{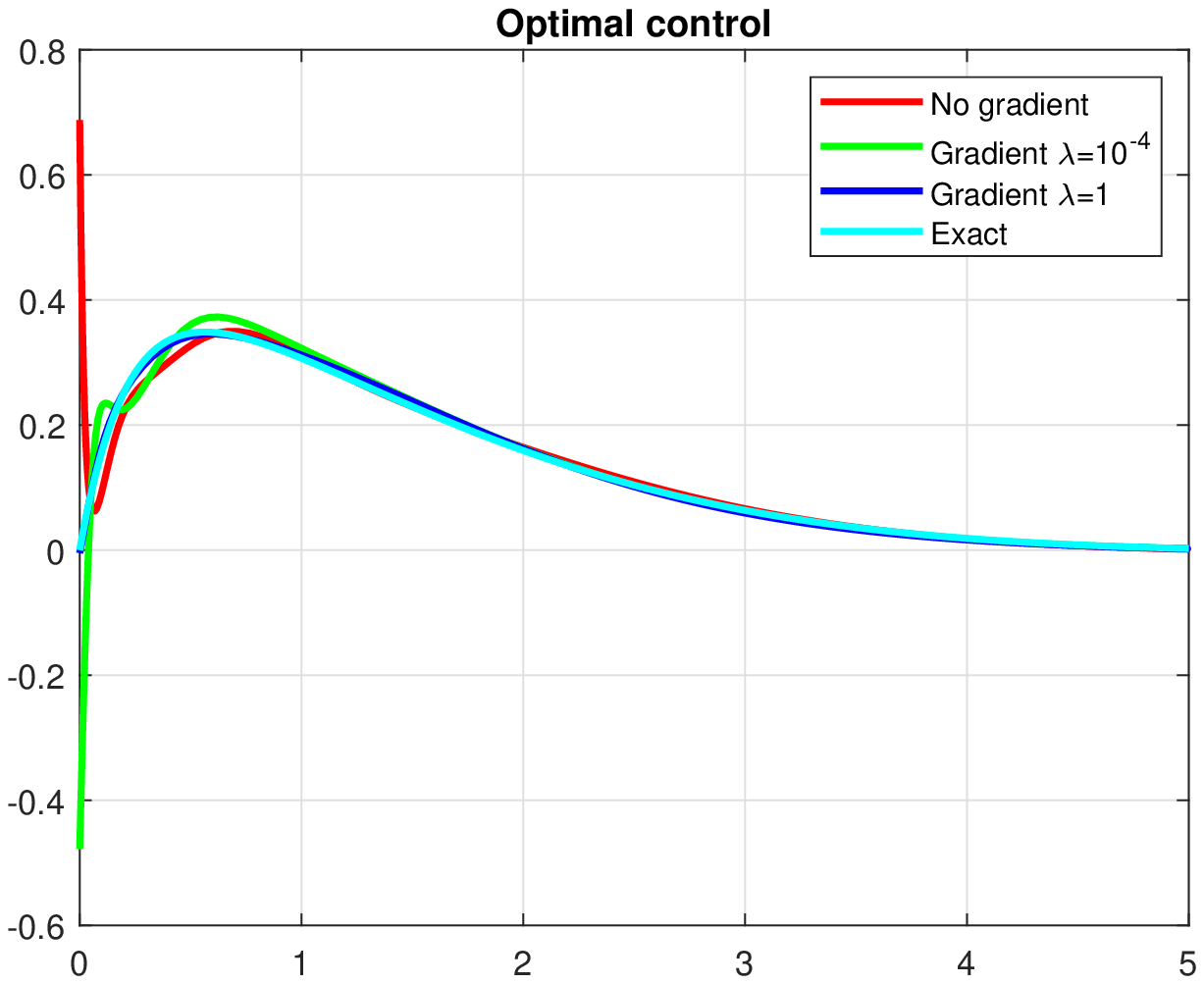}	
\caption{Optimal control without noise (left) and with noise amplitude $\sigma=10^{-2}$ (right) starting from $x_0=(1,-1)$.}
\label{Fig_test2D}
\end{figure}

\subsubsection*{Constrained control}

We focus on the optimal control problem \eqref{dyn2D}-\eqref{J2D} coupled with control constraints, $i.e.$ $|u|\le u_{max}$. As remarked in Section \ref{constraint}, in this case the value function and its gradient will be computed via the PMP, composing the optimal control as
$$
u^*=u_{max} \tanh\left(\frac{u}{u_{max}}\right)
$$
to enforce the control constraint and the optimisation of the cost functional \eqref{cost_const}.
We solve the problem in the domain $[-2,2]^2$ and we consider as initial condition $(x_1(0),x_2(0))=(2,2)$.
In the first test, we fix the TT-rank $r=5$.
In Table \ref{Table_constraints} we report the values of the total cost obtained using different $\lambda$ and different constraints. Both in the unconstrained case ($u_{max}=\infty$) and in the constrained cases it is possible to pick a $\lambda\neq 0$ which performs better than the case with $\lambda=0$.
It is not necessarily the largest $\lambda$, as an exceedingly large $\lambda$ deteriorates the conditioning of the normal equation \eqref{normal_eq}.
In the left panel of Figure \ref{traj_constraint} we show the optimal trajectories for $\lambda=0$ in the unconstrained case. In this case the optimal control manages to steer the dynamical system to the equilibrium.
However, if we constrain the control to $u_{max}=20$ (right panel of Figure \ref{traj_constraint}), we notice that the control is unable to stabilize the system. %  for all values of $\lambda$.

\begin{table}[htbp]							
		\centering
		\begin{tabular}{c c c c}
					\toprule
					& $u_{max}=\infty$ &$u_{max}=25$ &$u_{max}=20$
					     \\
					\midrule
					$\lambda=0$ &       70.2131           & 81.9607 &  93.5168 \\
					$\lambda=10^{-4}$ & \textbf{70.2124}  & 81.4676 & 92.0130 \\
					$\lambda=10^{-3}$ & 70.2139           & 81.1542 & \textbf{91.9481} \\
					$\lambda=10^{-2}$ & 70.2134           & \textbf{81.1451} &  92.0824  \\
			\bottomrule		
		\end{tabular}	

							\caption{Cost functional $\tilde J_T$ for different $\lambda$ and $u_{max}$ with $r=5$. In each column there is a nonzero $\lambda$ that gives the smallest cost. \label{Table_constraints}}
						
	\end{table}
	
	\begin{figure}[htbp]	
\centering
	\includegraphics[scale=0.5]{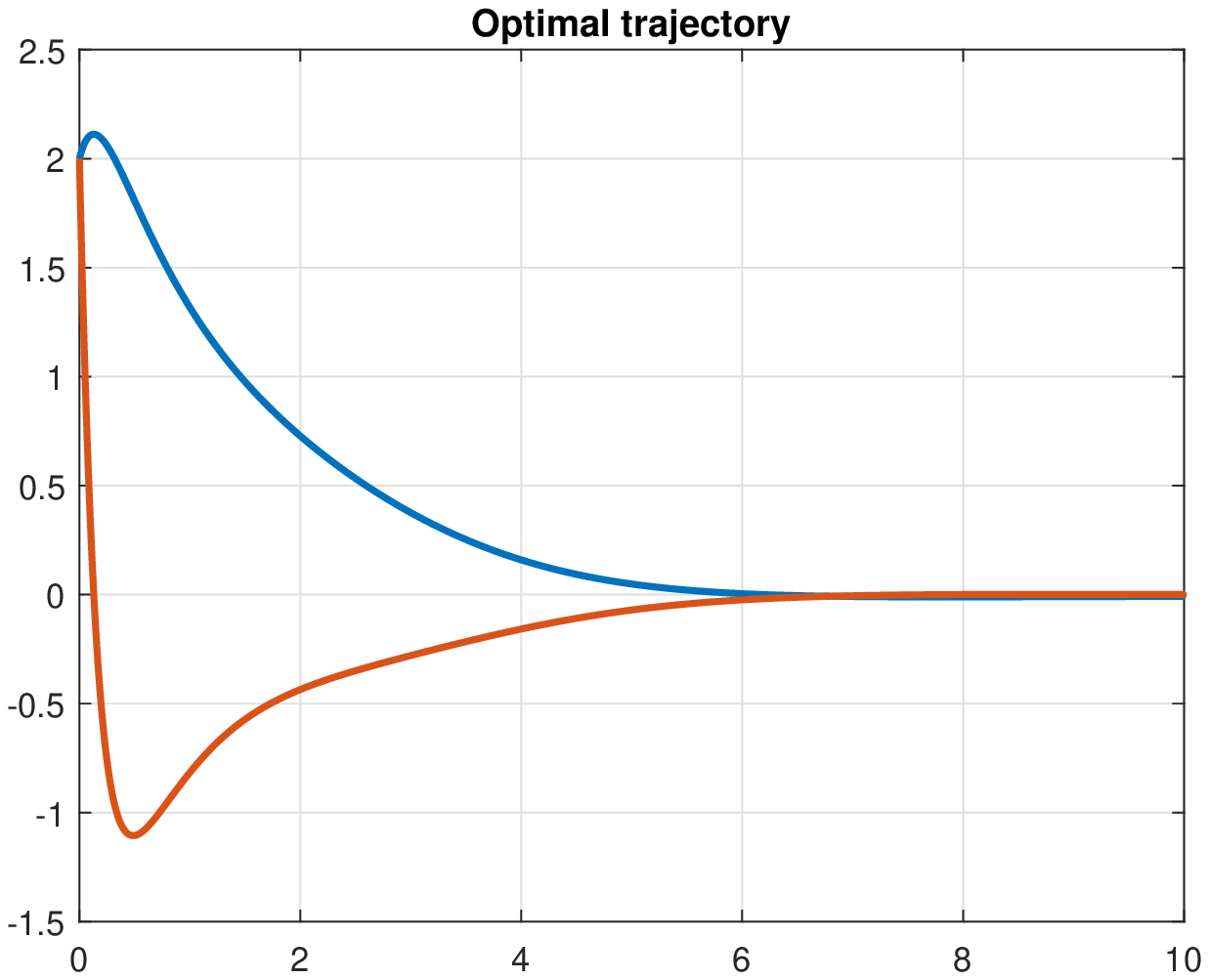}	
	\includegraphics[scale=0.5]{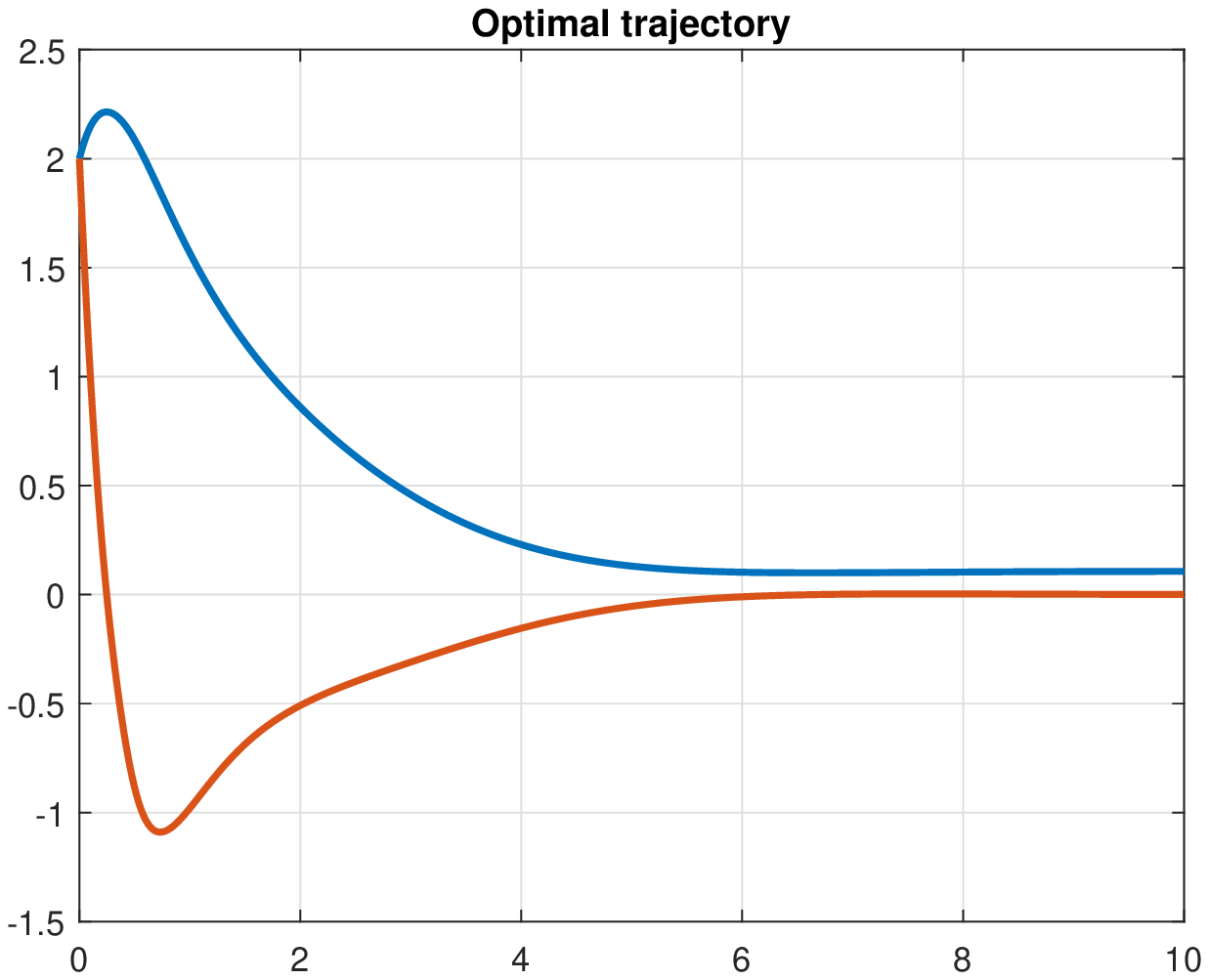}	
\caption{Optimal trajectory (blue: $x_1(t)$, red: $x_2(t)$) for $u_{max}=\infty$ (left) and $u_{max}=20$ (right) with $\lambda=0$ and TT rank $r=5$. This TT rank is too small to approximate the value function accurate enough to stabilize the trajectory.}
\label{traj_constraint}
\end{figure}

	\begin{table}[htbp]
		\centering
		\begin{tabular}{c c c c}
					\toprule
					& $u_{max}=\infty$ &$u_{max}=25$ &$u_{max}=20$
					     \\
					\midrule
					$\lambda=0$       & \textbf{70.2123} &  81.9629 &   93.2145 \\
					$\lambda=10^{-4}$ & 70.2124          & 81.4544 &  91.9991\\
					$\lambda=10^{-3}$ & 70.2130          &  81.1660 & \textbf{91.8594} \\
					$\lambda=10^{-2}$ & 70.2131          & \textbf{81.1548} & 92.0458  \\
			\bottomrule
		\end{tabular}

							\caption{Cost functional $\tilde J_T$ of the 2D model for different $\lambda$ and $u_{max}$ with $r=6$. \label{Table_constraints2}}

	\end{table}

This is fixed by increasing the TT-rank of our approximation to $r=6$.
We can see by the right panel of Figure \ref{traj_constraint2} that now the solution reaches the origin. In the left panel of Figure  \ref{traj_constraint2} we show the different behaviours of the control according to the different constraints, fixing $\lambda=10^{-3}$. Finally, we show in Table \ref{Table_constraints2} the total cost for the different choices of $\lambda$ and $u_{max}$ with $r=6$. Reducing the size of the constraint box, the difference between the no-gradient regression and choosing the best $\lambda$ increases, confirming that the gradient cross achieves a better result for the constrained case in presence of information of both the value function and its gradient.

	\begin{figure}[htbp]	
\centering
	\includegraphics[scale=0.5]{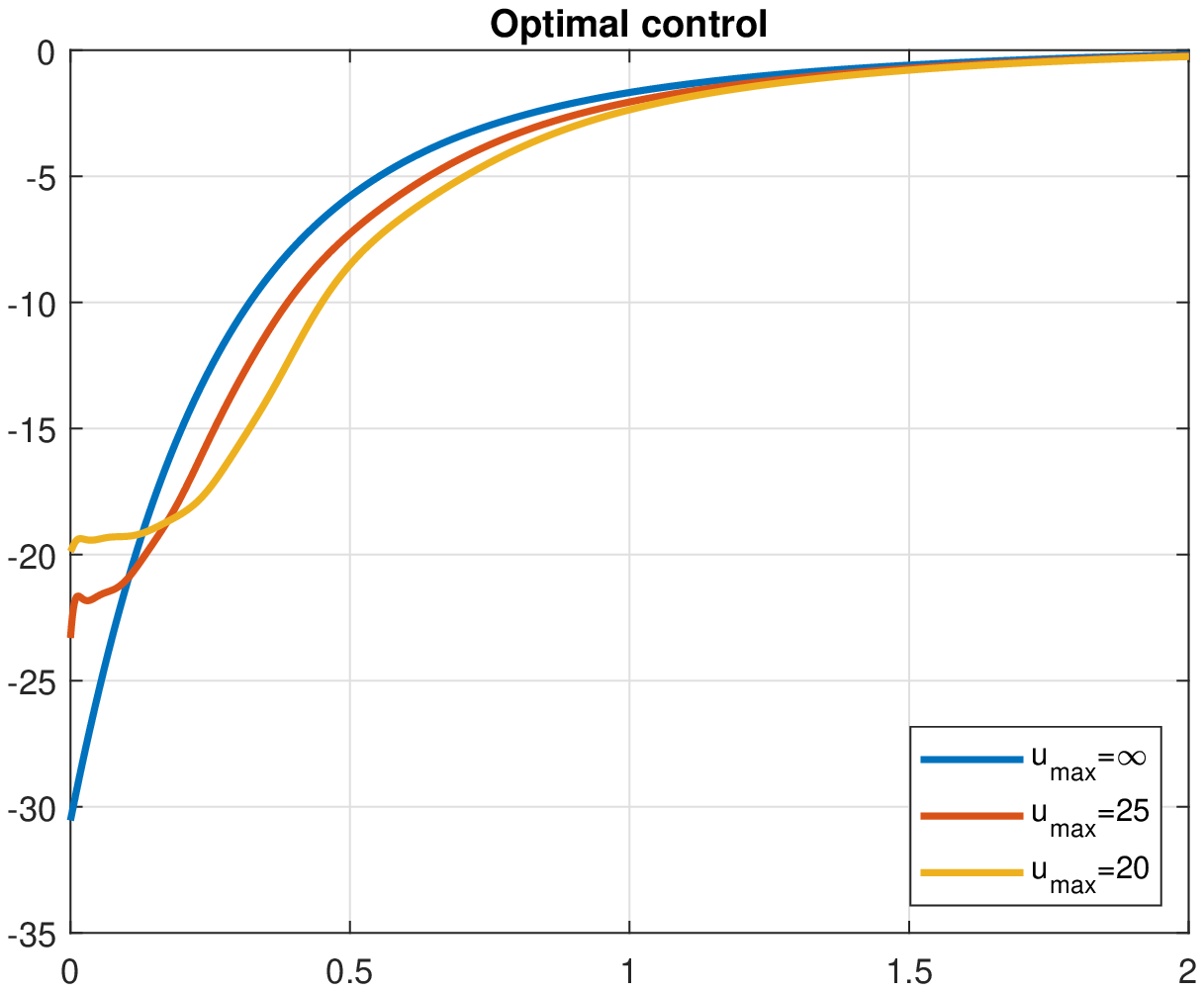}	
	\includegraphics[scale=0.5]{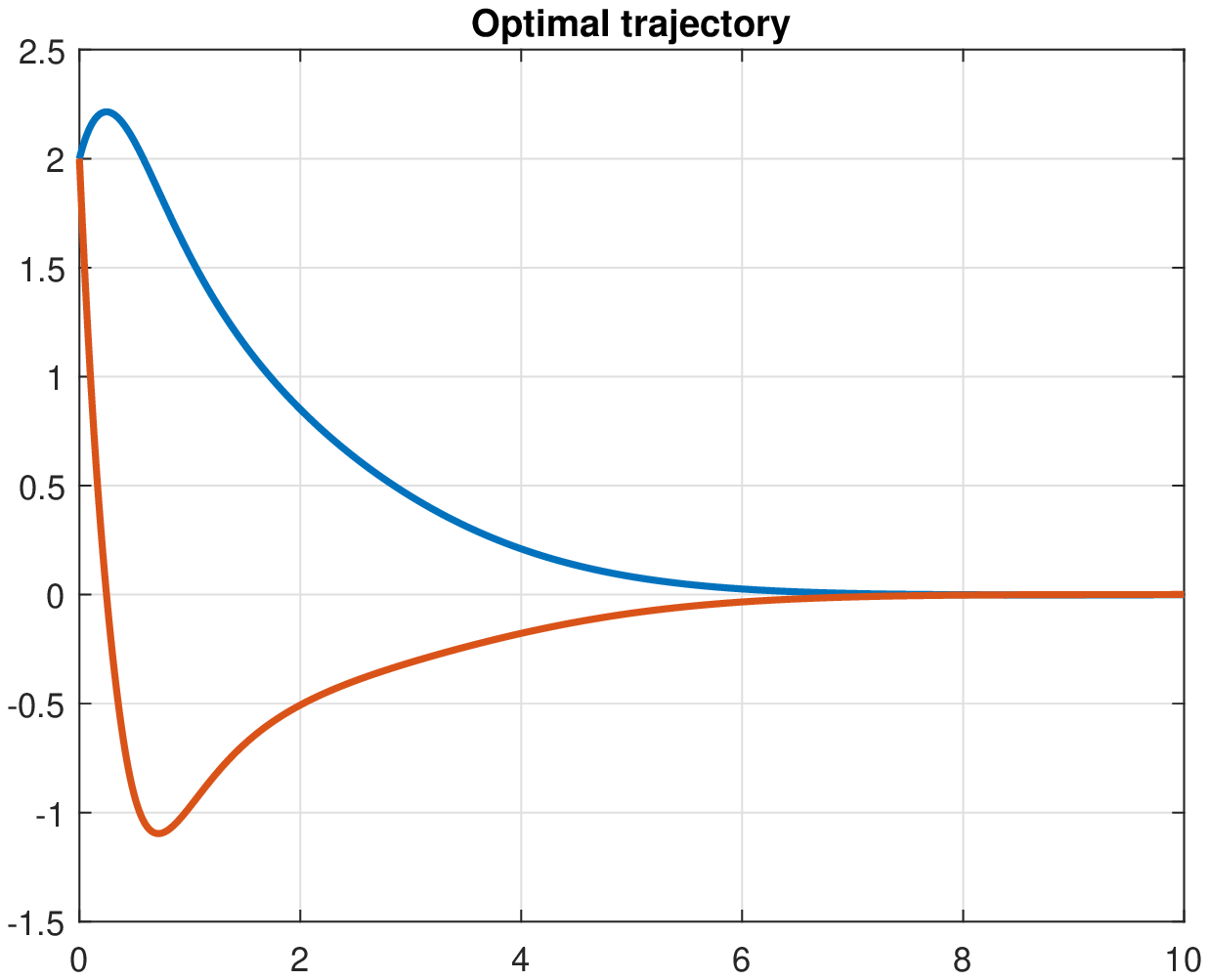}	
\caption{Optimal control for different $u_{max}$ (left) and optimal trajectory (blue: $x_1(t)$, red: $x_2(t)$) of the 2D model for $u_{max}=20$ (right) with $\lambda=10^{-3}$ and $r=6$. This TT rank is sufficient to stabilize the system.}
\label{traj_constraint2}
\end{figure}
%	\begin{table}[htbp]							
%		\centering
%		\begin{tabular}{c c c c}
%					\toprule
%					& $u_{max}=\infty$ &$u_{max}=25$ &$u_{max}=20$
%					     \\
%					\midrule
%					$\lambda=0$ & 70.2123 & 74.5241 &   80.6042 \\
%					$\lambda=10^{-4}$ & 70.2124  & 73.6890 &  78.5408 \\
%					$\lambda=10^{-3}$ & 70.2130 & 72.6739 & 76.5195 \\
%					$\lambda=10^{-2}$ & 70.2131 & 72.4734 & 77.8683  \\
%			\bottomrule		
%		\end{tabular}	
%
%							\caption{Cost functional for different $\lambda$ and $u_{max}$ with $r=6$. \label{Table_constraints2}}
%						
%	\end{table}

\subsection{Lorenz system}

The third example deals with the Lorenz system given by
\begin{equation}
\begin{cases}
\dot{x}= \sigma (y-x), \\
\dot{y}= x(\rho-z)-y+u, \\
\dot{z}=xy-\beta z
\end{cases}
\end{equation}
with the following cost functional
\begin{equation}
J= \int_0^\infty \left( |x(s)|^2+|y(s)|^2+|z(s)|^2 + \gamma |u(s)|^2 \right) \, ds.
\label{cost_lorentz}
\end{equation}
The same example has been considered in \cite{KK21}. We fix $\sigma=10$, $\beta=8/3$ , $\rho=2$ and $(x(0),y(0),z(0))=(-1,-1,-1)$.
Data collection is performed via SDRE and we will consider two cases: $\lambda=0$ and $\lambda=1$. We will vary the regularization parameter $\gamma$ in \eqref{cost_lorentz} in the range $\{0,0.1,0.01,0.001\}$.
We consider 6 Legendre basis functions on the interval $[-1,1]$ and the stopping error threshold for the gradient cross is $10^{-2}$.

In Figure \ref{traj_contr} we show the optimal trajectory and the optimal control with $\gamma=0.001$ and $\lambda=1$.
In Figure \ref{time_steps} we show the number of time steps needed for \texttt{ode45} solver to compute the optimal trajectory in logarithmic scale. It is evident that for small values of the regularization parameter $\gamma$, the ODE solver needs more time steps in the no-gradient case compared to the gradient case. The norm of the difference of the two value functions is $1.2\cdot 10^{-6}$, and the difference of the total cost functionals is $3.5\cdot 10^{-7}$, which is within the requested error threshold. However, the value function approximation computed with $\lambda=1$ appears smoother, allowing to obtain a similar trajectory with fewer time steps.

	\begin{figure}[htbp]	
\centering
	\includegraphics[scale=0.5]{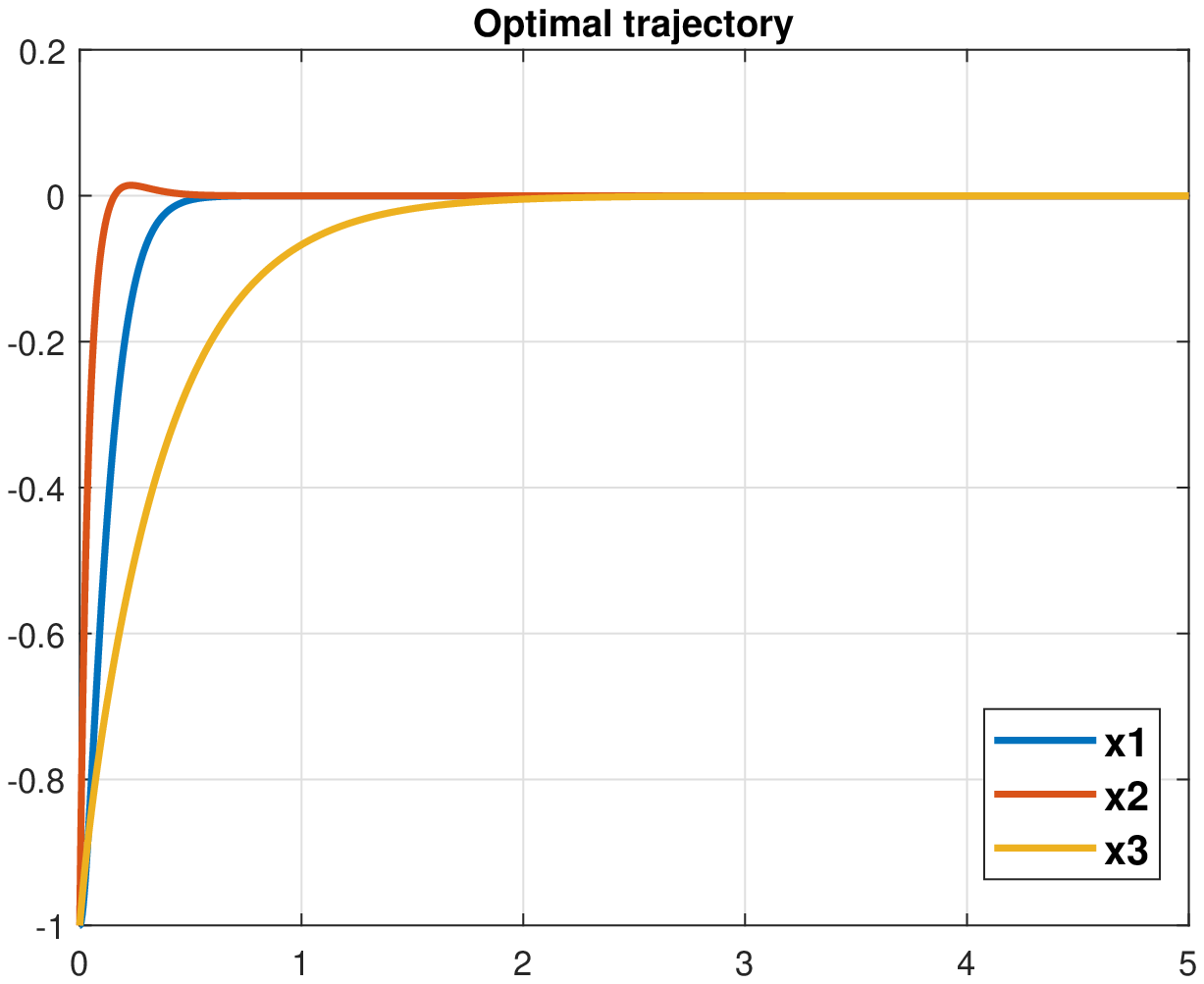}	
	\includegraphics[scale=0.5]{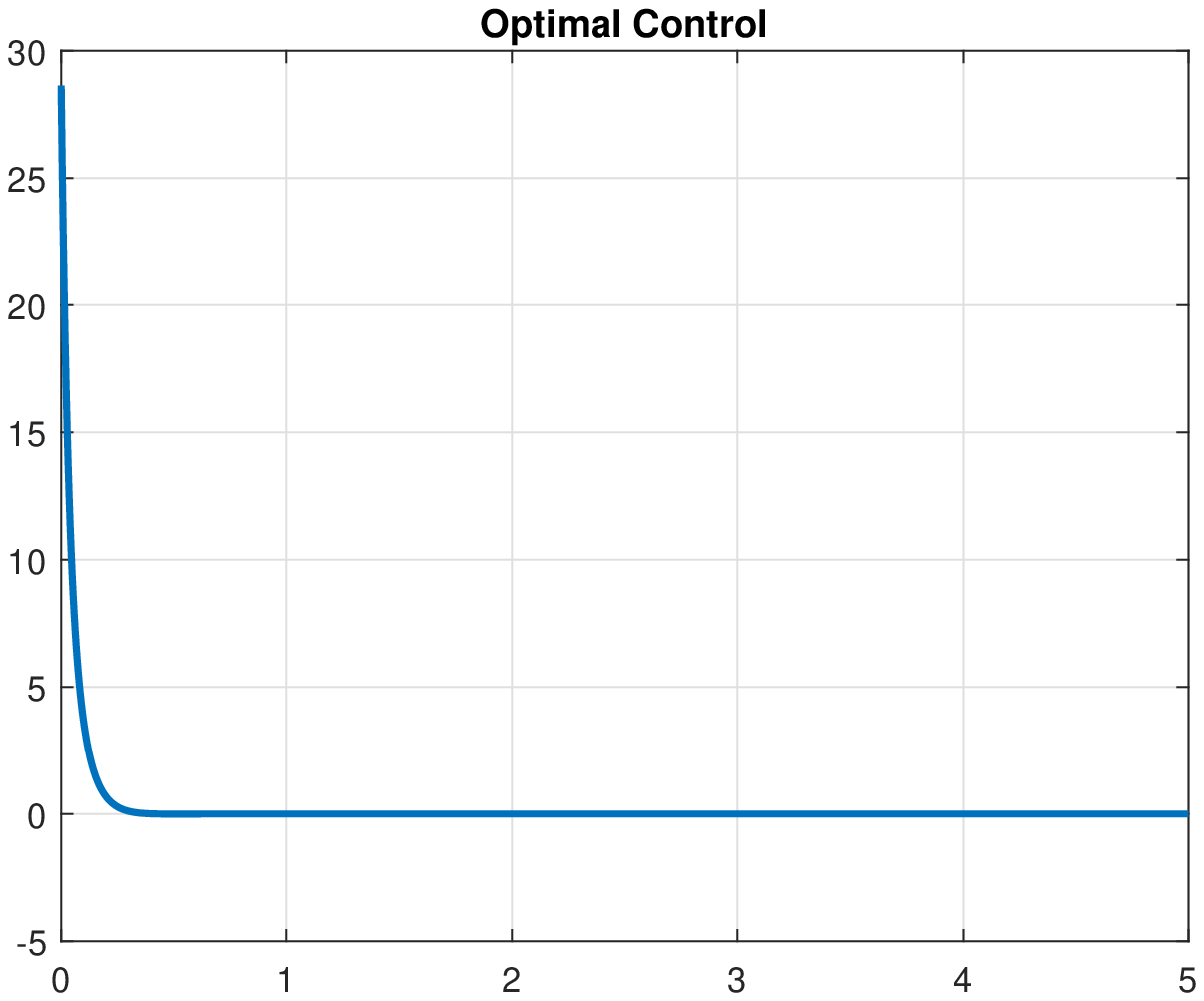}	
\caption{Optimal trajectory (left) and optimal control (right) of the Lorenz model with $\gamma=0.001$ and $\lambda=1$.}
\label{traj_contr}
\end{figure}
	\begin{figure}[htbp]	
\centering
	\includegraphics[scale=0.5]{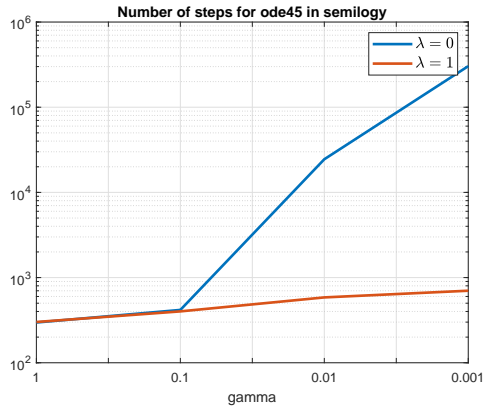}	
\caption{Numbers of time steps needed for the computation of the controlled trajectory of the Lorenz model using \texttt{ode45} for different regularization parameters $\gamma$. For large values of $\gamma$ the two curves overlap, while we note an increasing difference for small $\gamma$. For the smallest $\gamma=10^{-3}$, the no-gradient method needs $3$ orders of magnitude more time steps compared to the gradient method.}
\label{time_steps}
\end{figure}

\subsection{Cucker-Smale model}

Let us consider the dynamics governed by the Cucker-Smale model with $N_a$ interacting agents given by
\begin{equation}
\begin{bmatrix} \dot{y} \\ \dot{v}  \end{bmatrix} = \begin{bmatrix} \mathbb{O}_{N_a} & \mathbb{I}_{N_a} \\ \mathbb{O}_{N_a} & \mathcal{A}_{N_a}(y)  \end{bmatrix} \begin{bmatrix} y \\ v  \end{bmatrix} + \begin{bmatrix} \mathbb{O}_{N_a} \\ \mathbb{I}_{N_a}  \end{bmatrix} u,
\end{equation}
with
$$
\left[ \mathcal{A}_{N_a}(y) \right]_{i,j}= \begin{cases} -\frac{1}{N_a} \sum_{k \neq i} P(y_i,y_k) & if \;i=j, \\
\frac{1}{N_a} P(y_i,y_j) & otherwise,
						 				 \end{cases}
$$
$$
P(y_i,y_j)=\frac{1}{1+\Vert y_i-y_j \Vert^2} \,.
$$
Our aim is to minimize the following cost functional
$$
J(y(\cdot),v(\cdot),u(\cdot))= \frac{1}{N_a} \int_0^\infty \Vert y(s) \Vert^2 + \Vert v(s) \Vert^2 + \Vert u(s) \Vert^2\, ds.
$$
We consider a state domain $\Omega=[-0.5,0.5]^{2N_a}$, $5$ Legendre basis functions in each variable, and the gradient cross stopping tolerance is equal to $10^{-2}$.
We first compare PMP and SDRE to check their performances in producing samples of the cost.
In this case we fix $N_a=2$ and final time $T=20$ for the corresponding finite horizon control problem for PMP. The PMP system \eqref{pmp} is solved via the MATLAB function \texttt{bvp4c}. It is possible to supply this function with the Jacobian of the differential equation to accelterate the algorithm and to obtain a more accurate solution. We will denote by PMP the resolution of the system \eqref{pmp} without the the knowledge of the Jacobian, while by PMPJ the one enriched by this further information. The results of the comparison are shown in Table \ref{Table_comparison}.

	\begin{table}[htbp]							
		\centering
		\begin{tabular}{c c c c}
					\toprule
					& CPU & $\tilde J_T$ & $\tilde y_{\max}(T)$
					     \\
					\midrule
					SDRE & $0.5s$ & 0.150168 &    1.5e-8 \\
					PMP & $33s$  & 0.150173 &  2.0e-6 \\
					PMPJ & $24s$  & 0.150173 &  6.2e-7 \\

			\bottomrule		
		\end{tabular}	

							\caption{Comparison between SDRE, PMP and PMPJ on the Cucker-Smale model with $N_a=2$. Here SDRE is the best method. \label{Table_comparison}}
						
	\end{table}

We can notice that SDRE is 66 times faster than PMP and 48 times faster compared to PMPJ, obtaining almost the same result in terms of the total cost, and a much smaller absolute value of the final state than PMPJ and PMP. In the following run we use the SDRE approach to generate the data for the TT Gradient Cross.

We turn our attention to higher dimensional problems. We first analyse the behaviour of the Gradient Cross algorithm increasing the dimension $d=2N_a$. In Figure \ref{Fig_shade1} we compare the error in the cost functional and the maximum reached by the dynamics at the final time increasing the dimension $d$ from 4 to 20. We consider two cases, one in the absence of gradient information ($\lambda=0$) and the other one with $\lambda=10^{-3}$.
The shaded areas in the plots are encircled by mean $\pm$ 1 standard deviation over 10 trials.
For example in the case of the error in the cost we consider the mean $\overline{err}_J(d)$ and the corresponding standard deviation $\sigma_{err_J}(d)$. The shaded area is created considering for each dimension the interval $[\overline{err}_J(d)-\sigma_{err_J}(d),\overline{err}_J(d)+\sigma_{err_J}(d)]$. We see that the mean and the standard deviation are both lower in the case with gradient information and this difference grows with increasing dimension of the problem. In Figure \ref{Fig_shade2} we show the comparison in terms of TT ranks and evaluations needed by the Gradient Cross. It is important to point that the TT ranks fluctuate in the same interval $[11.5,~16]$ for all dimensions, proving that we are really solving the curse of dimensionality. In terms of the number of evaluations, for lower dimensions the case with $\lambda=10^{-3}$ presents a higher mean, but for higher dimensions we obtain a decreasing behaviour in contrast to the case with $\lambda=0$.

	\begin{figure}[H]	
\centering
	\includegraphics[scale=0.5]{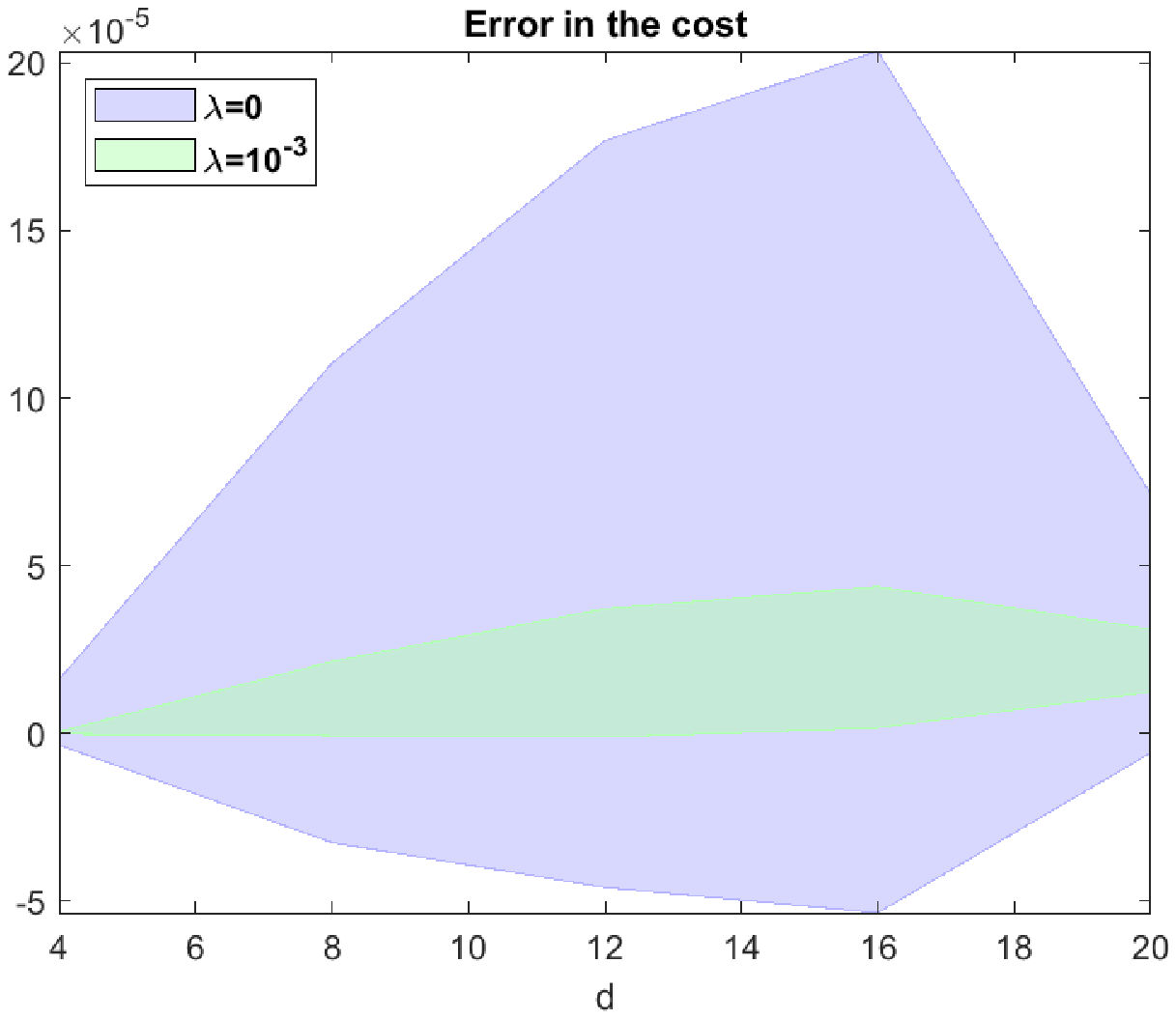}	
	\includegraphics[scale=0.5]{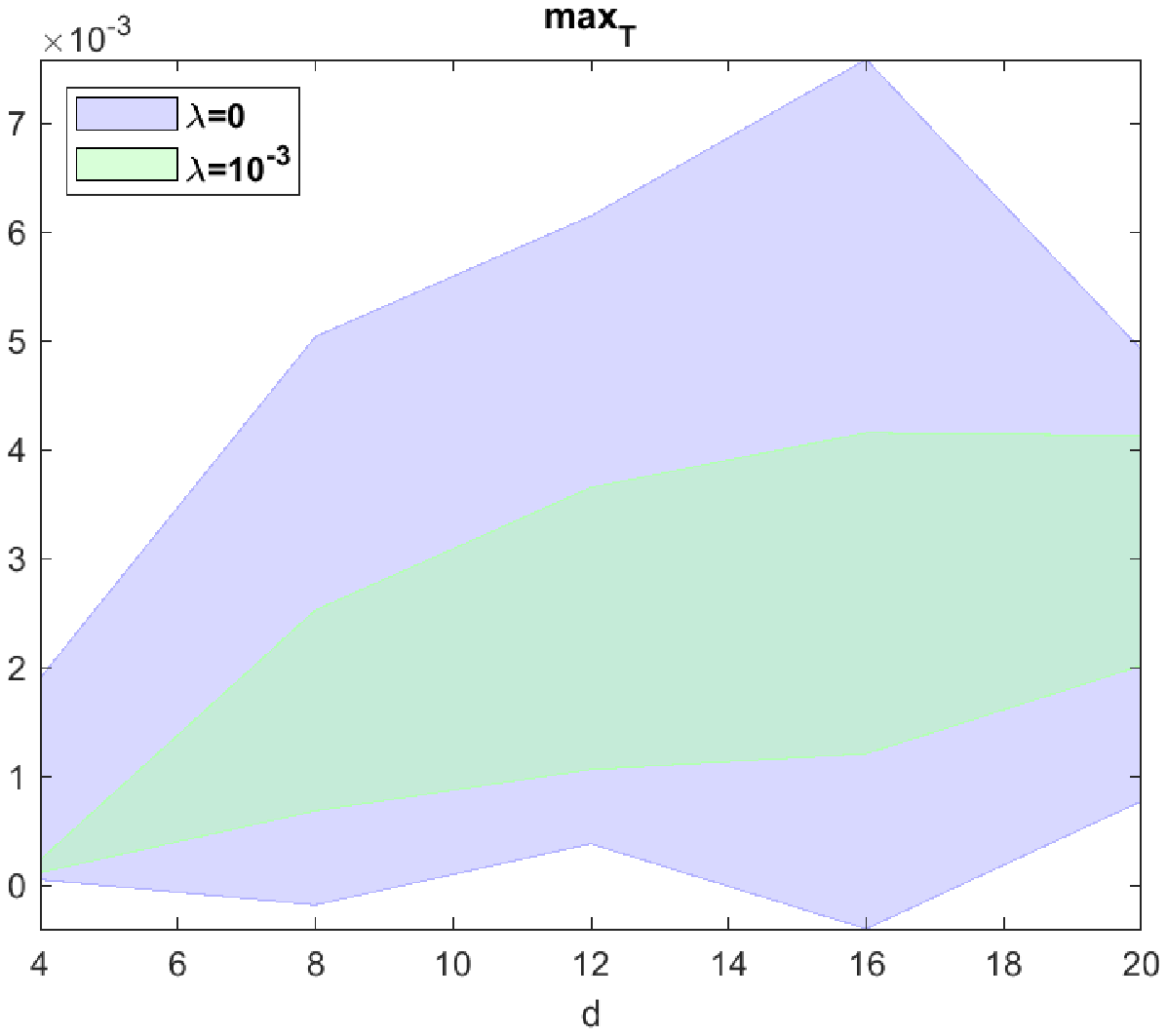}	
\caption{Error in the Cucker-Smale cost functional (left) and $\tilde y_{\max}(T)$ (right) for $\lambda=0$ (blue) and $\lambda=10^{-3}$ (green). Shaded area denotes mean $\pm$ 1 standard deviation over $10$ runs. Here $\lambda=10^{-3}$ gives a more accurate approximation.}
\label{Fig_shade1}
\end{figure}

	\begin{figure}[H]	
\centering
	\includegraphics[scale=0.5]{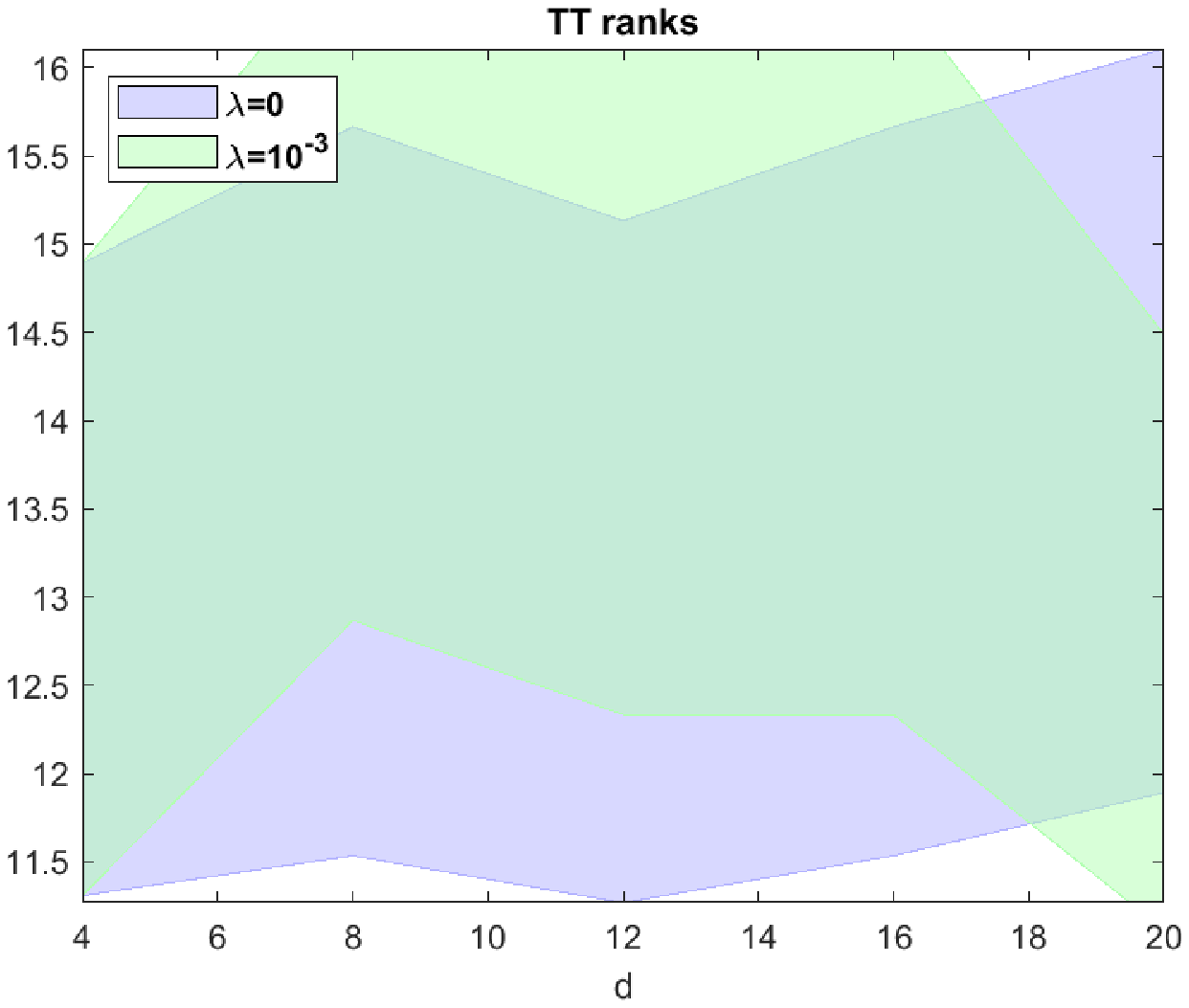}	
	\includegraphics[scale=0.5]{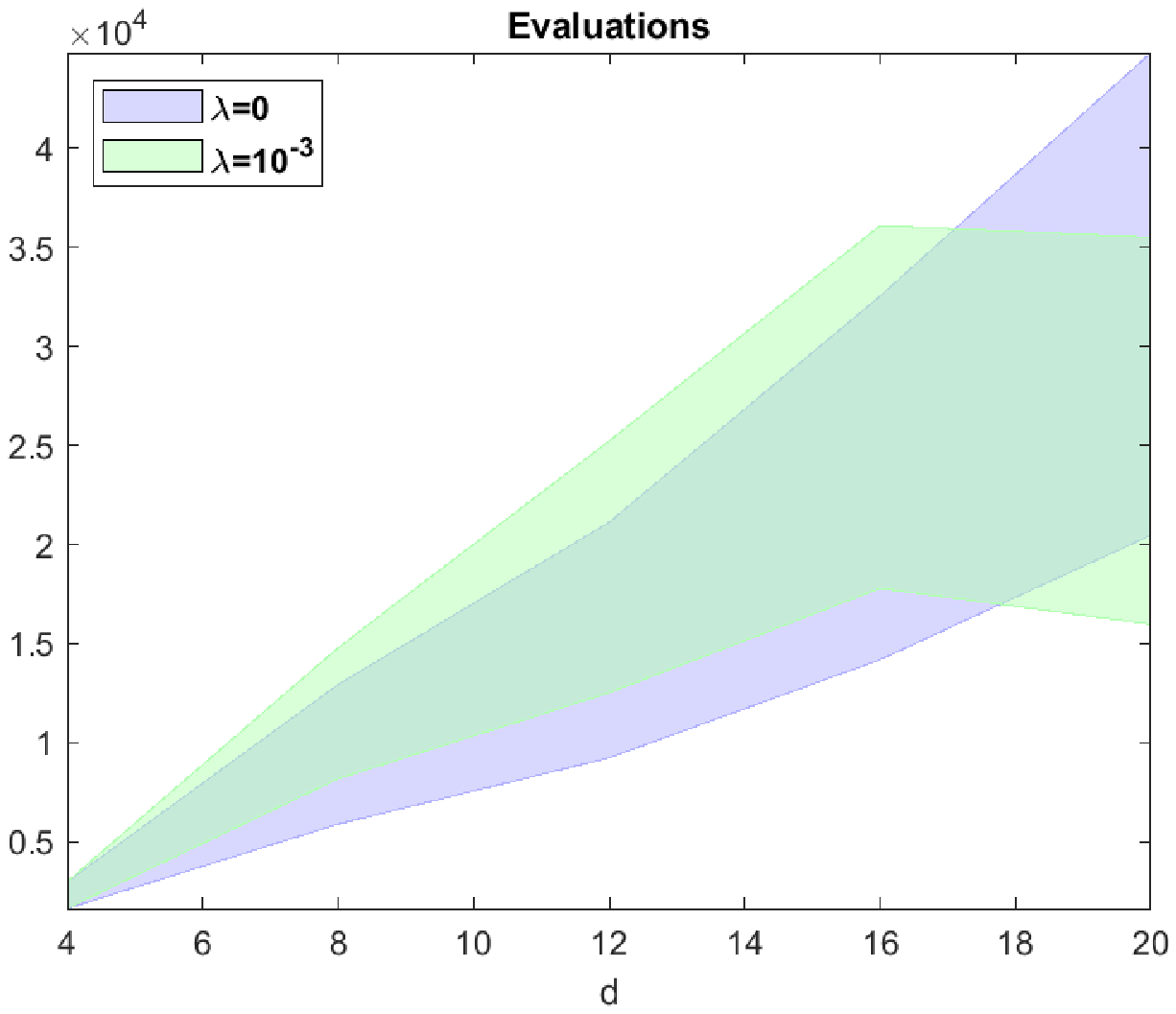}	
\caption{TT ranks (left) and number of evaluations (right)  for $\lambda=0$ (blue) and $\lambda=10^{-3}$ (green). Shaded area denotes mean $\pm$ 1 standard deviation over $10$ runs. Both methods show comparable complexity, though $\lambda=10^{-3}$ has less variation from run to run.}
\label{Fig_shade2}
\end{figure}

The choice of the parameter $\lambda$ is crucial in this approach. We show a comparison of the performances for different $\lambda$ and in different dimensions. Since the computational costs of our previous tests  are comparable, we will focus instead on the error in the cost functional and on $\tilde y_{\max}(T)$.
We consider a finite set $\Lambda$ for the variable $\lambda$ and we minimize the total cost and the final maximum value on this set. The result of the minimization will provide the best choice for the parameter $\lambda$. In Figure \ref{Fig_teta} we report the results for these quantities. The parameter $\lambda$ is taken from the set $\Lambda=\{0\} \cup \{10^{-k}, k=0,\ldots, 6\}$, and the dimension $d$ varies in the range $\{10,20,30,40\}$. The minimum for each dimension is marked with a circle. First of all, we can notice that in all the cases the parameter $\lambda=0$ never represents the optimal choice. We can notice that we obtain a parameter independent on the dimension since in almost all cases $\lambda=10^{-6}$ represents the optimal choice.
	\begin{figure}[htbp]	
\centering
	\includegraphics[scale=0.5]{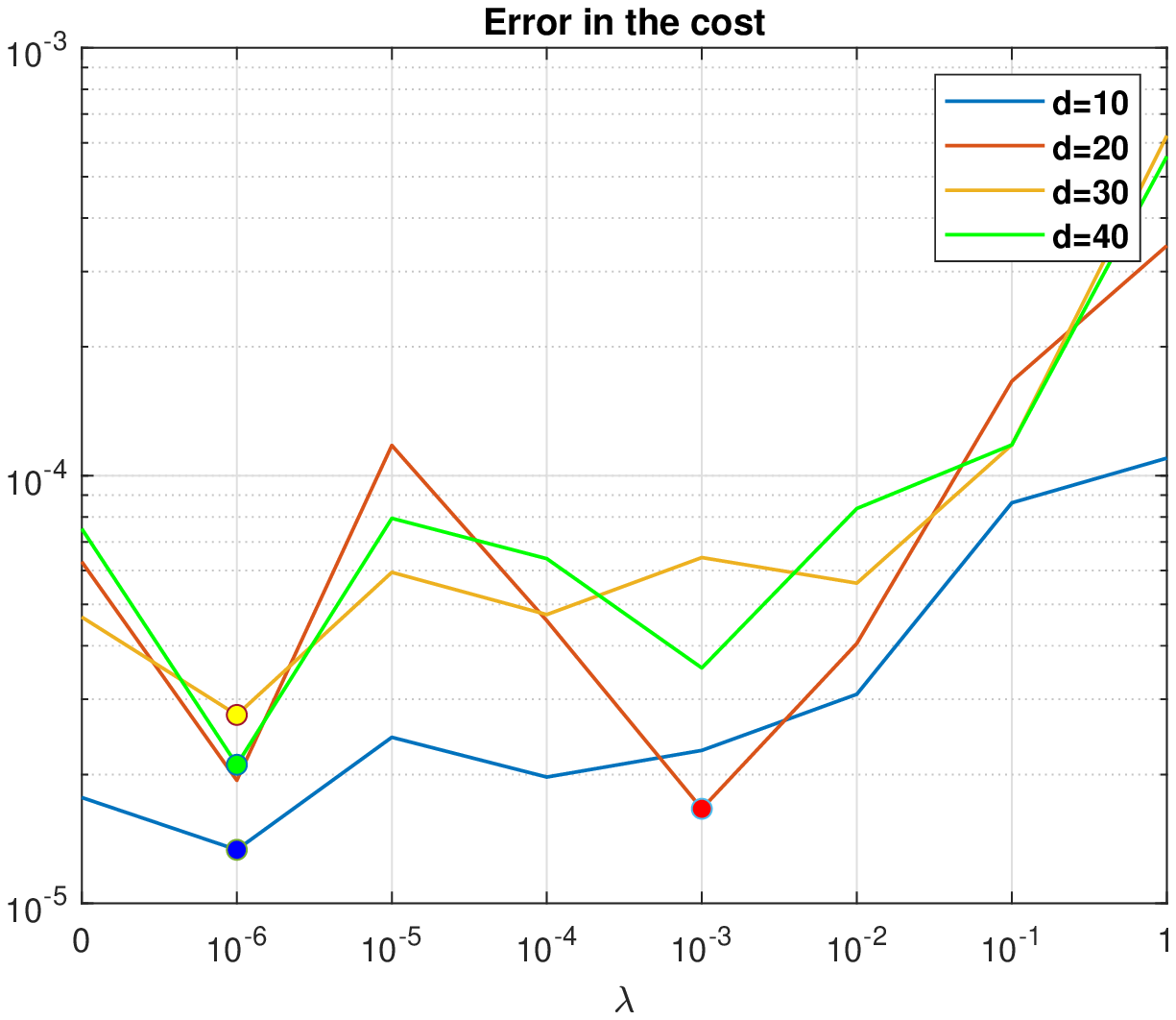}	
	\includegraphics[scale=0.5]{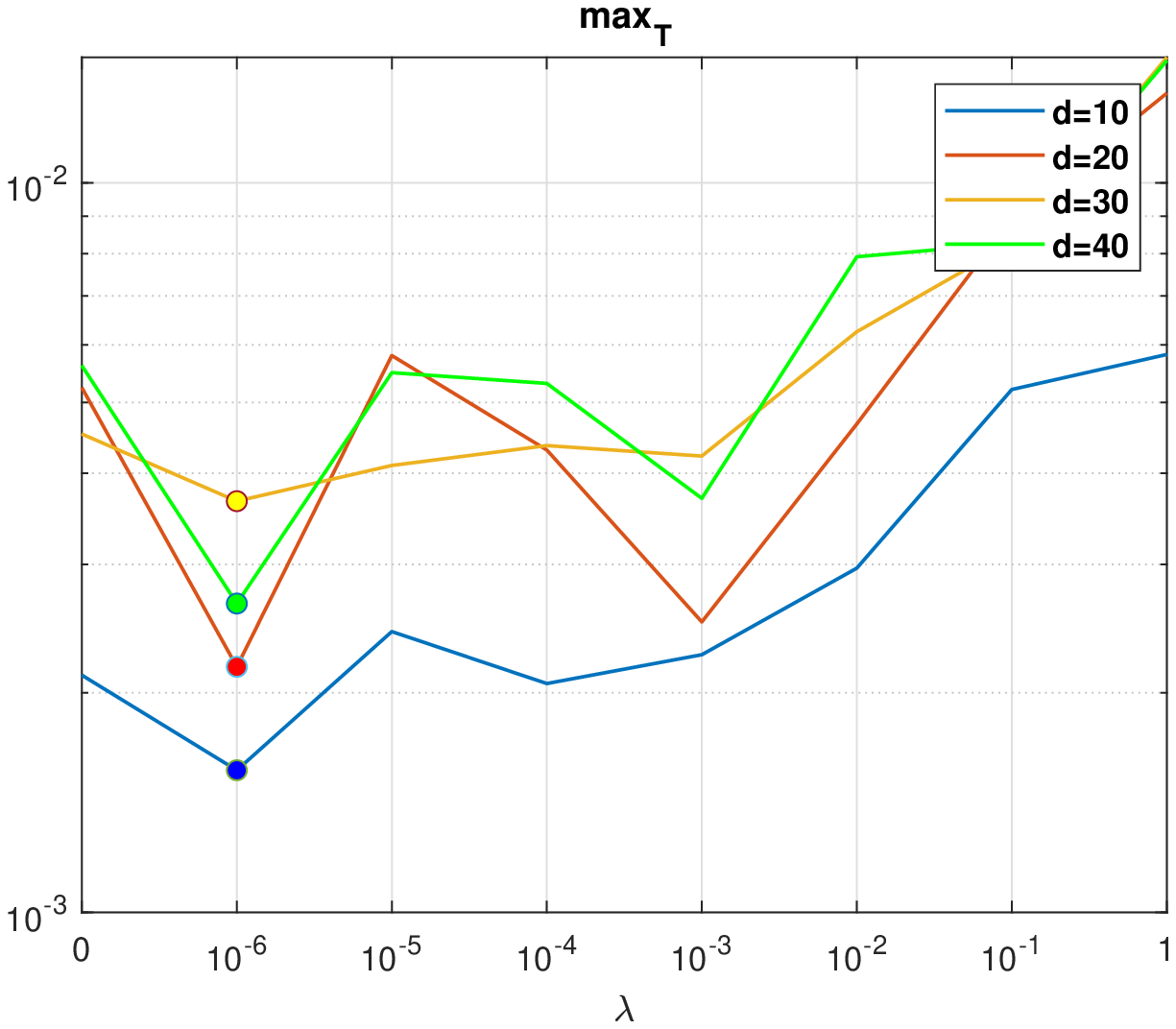}	
\caption{Error in the Cucker-Smale cost functional (left) and $\tilde y_{\max}(T)$ (right) for different $\lambda$ and dimensions. It is possible to notice that $\lambda=10^{-6}$ represents the optimal choice in almost all cases studied.}
\label{Fig_teta}
\end{figure}

In Table \ref{Table_cpu_sdre} we report the averaged elapsed time in computing the suboptimal control applying directly SDRE (first column), via the value function precomputed by the TT (second column) and via the value function precomputed by the Two Boxes approach (third column). We immediately note that the evaluation of TT is 2 orders of magnitude faster than the online SDRE solution, proving the efficiency in precomputing the value function when a real-time solution is needed. By comparing the final two columns we notice a small speed-up, showing the beneficial application of LQR in a region close to the origin.

	\begin{table}[htbp]							
		\centering
		\begin{tabular}{c | c c c}
					\toprule
				$d$	& SDRE & TT & Two Boxes
					     \\
					\midrule
					10 &  $1.4e-3s$  & $ 1.8e-5s$   & $1.7e-5s$  \\
					20 &  $5.4e-3s$   & $6.9e-5s$ &   $6.4e-5s$ \\
					30 & $1.0e-2s$   &  $1.6e-4s$ &  $1.4e-4s$ \\
					40 & $2.2e-2s$   & $3.3e-4s$ &  $1.9e-4s$  \\
                   100 & $1.3e-1s$ & $5.3e-3s$ &  $4.5e-3s$ \\
			\bottomrule		
		\end{tabular}	

							\caption{ Averaged CPU time for a single computation of the suboptimal control for the different methods. \label{Table_cpu_sdre}}
						
	\end{table}

\subsubsection*{Comparing with Neural Networks}

The aim of this section is to compare the proposed technique with a supervised learning approach discussed in \cite{ABK21}. In this work the authors generate a dataset using a SDRE approach and train a neural network to learn directly a suboptimal feedback map $u(x)$. We will compare  against this approach in the framework of agent-based dynamics. 

In Figure \ref{Fig_TT_NN} we compare the two approaches varying the dimension $d \in \{10,20,30,40\}$. The choice of the parameter $\lambda$ for TT follows the previous paragraph. For the NN approach we consider the number of samples equal to the number of evaluations needed for TT, in this way the two strategies will have the same computational complexity. We can deduce that the TT approach is up to one order of magnitude more accurate in terms of both indicators.

			\begin{figure}[htbp]	
\centering
	\includegraphics[scale=0.5]{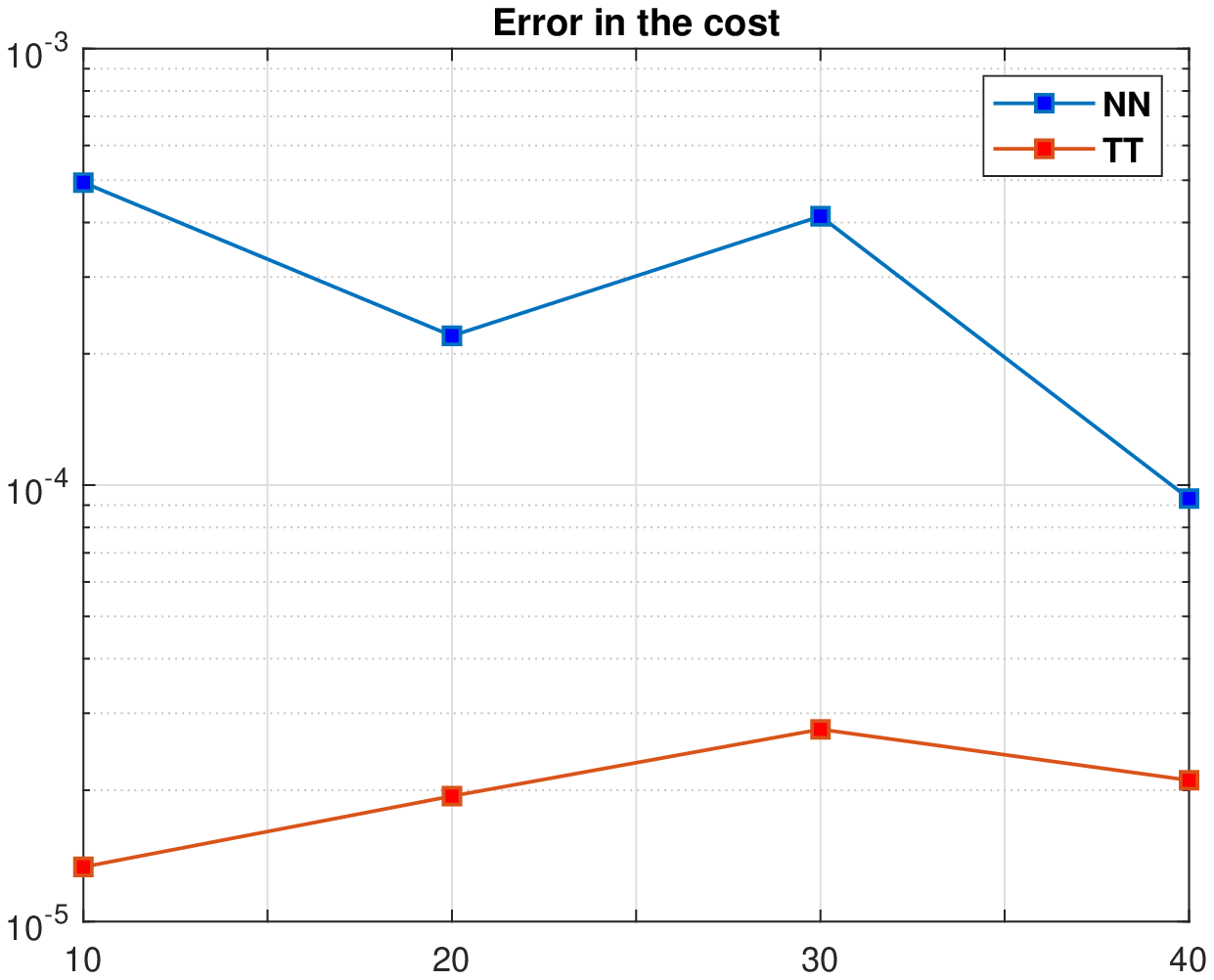}	
	\includegraphics[scale=0.5]{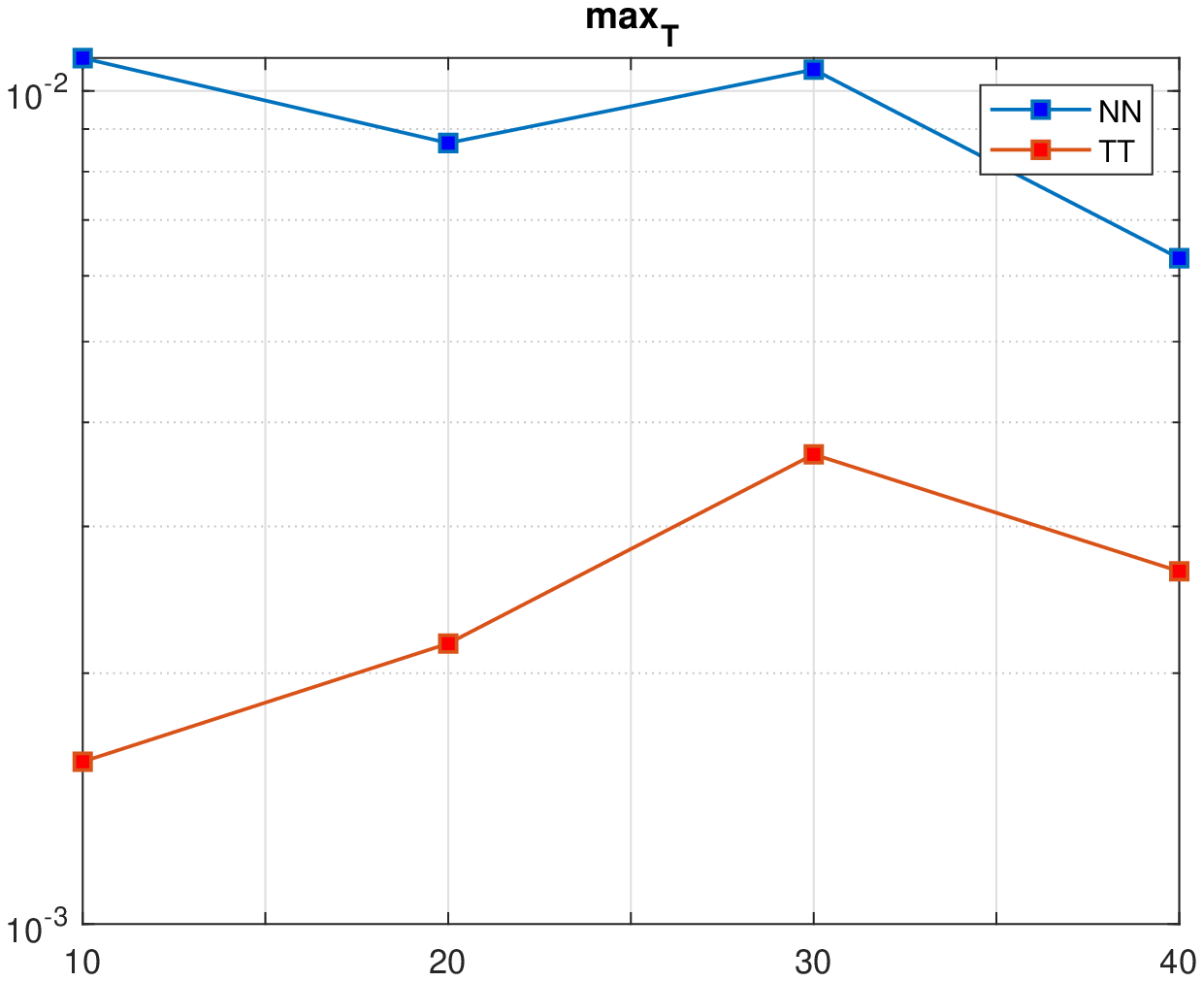}	
\caption{Error in the cost functional (left) and $\tilde y_{\max}(T)$ (right) for TT and NN for different dimensions. The TT approximation is more accurate in terms of both indicators for all dimensions.}
\label{Fig_TT_NN}
\end{figure}

Now we fix the dimension $d=40$ and we show the optimal trajectories for the two approaches in Figure \ref{Fig_na20}. As noticed in this figure and in the right panel of Figure \ref{Fig_TT_NN}, the NN is far from the equilibrium, while TT is evidently close. In Figure \ref{Fig_na20_zero} we show the optimal trajectories starting from $x_0=\underline{0}\in \mathbb{R}^{40}$ for TT (left panel) and NN (right panel). We note that all the components deviate from the equilibrium. The first $N_a$ components stabilize around a point different from the origin, while the last $N_a$ components return to 0.

			\begin{figure}[htbp]	
\centering
	\includegraphics[scale=0.5]{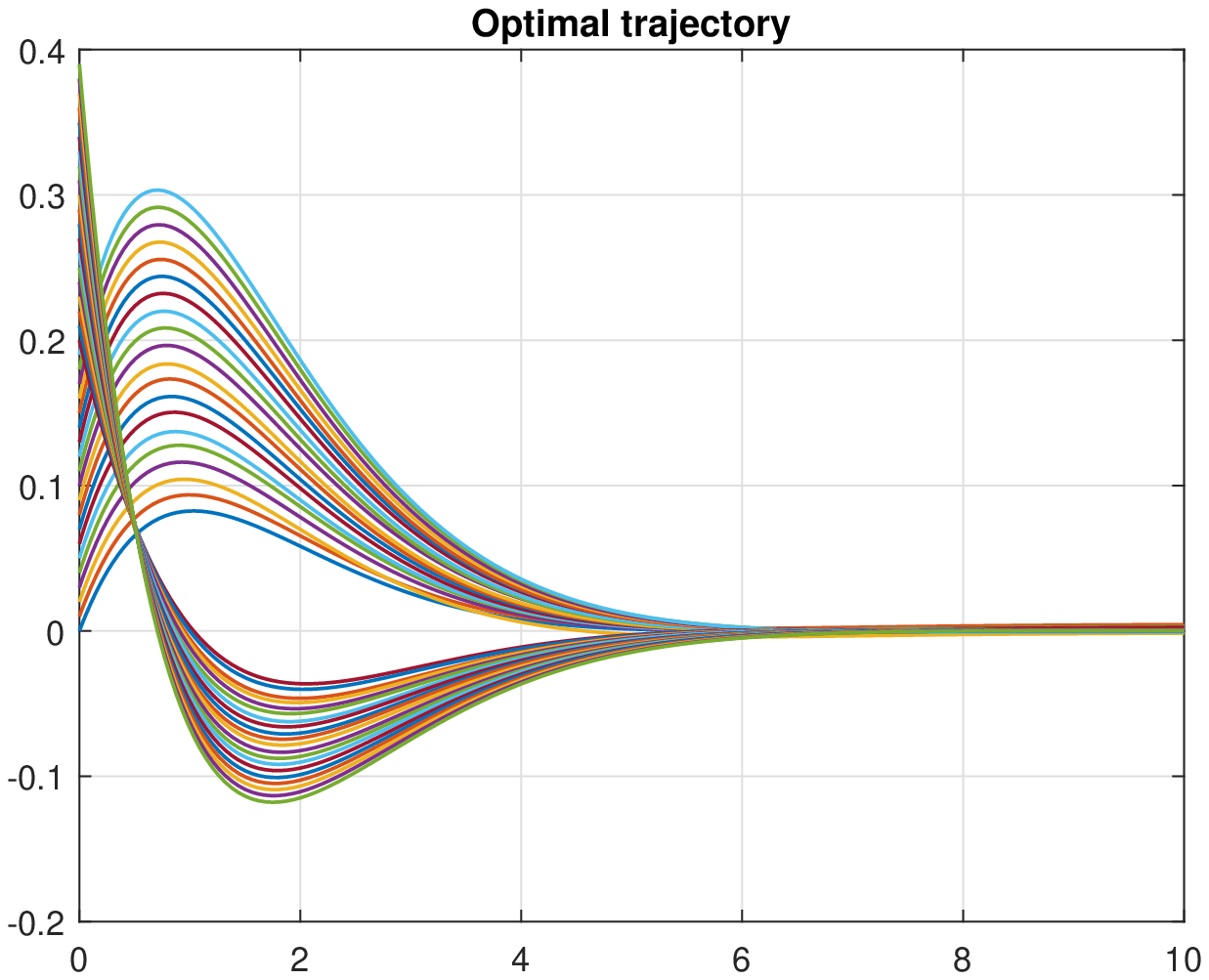}	
	\includegraphics[scale=0.5]{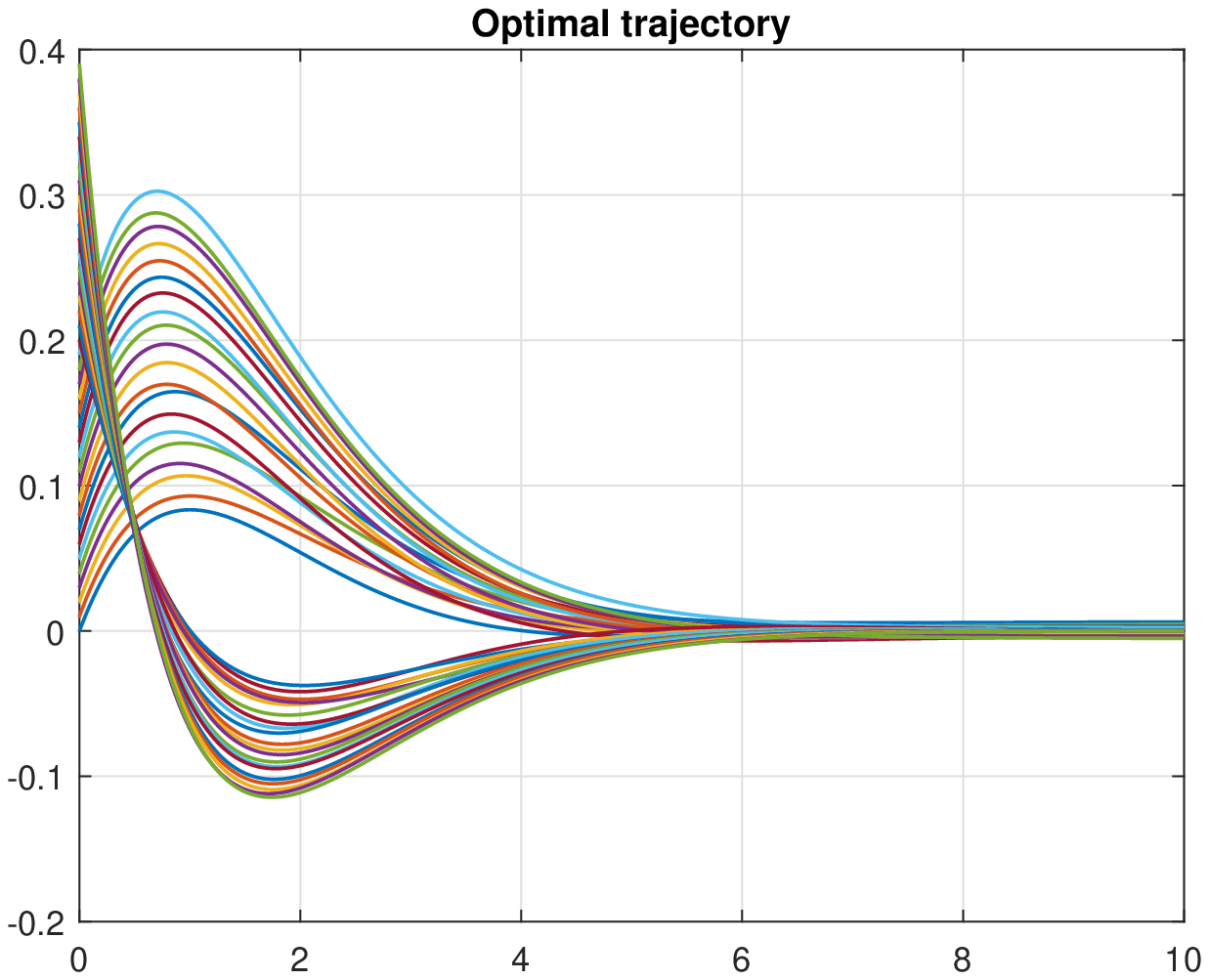}	
\caption{Optimal trajectory for TT (left) and NN (right) with $d=40$. Neither of the methods stabilizes the trajectory exactly towards the origin.}
\label{Fig_na20}
\end{figure}

			\begin{figure}[htbp]	
\centering
	\includegraphics[scale=0.5]{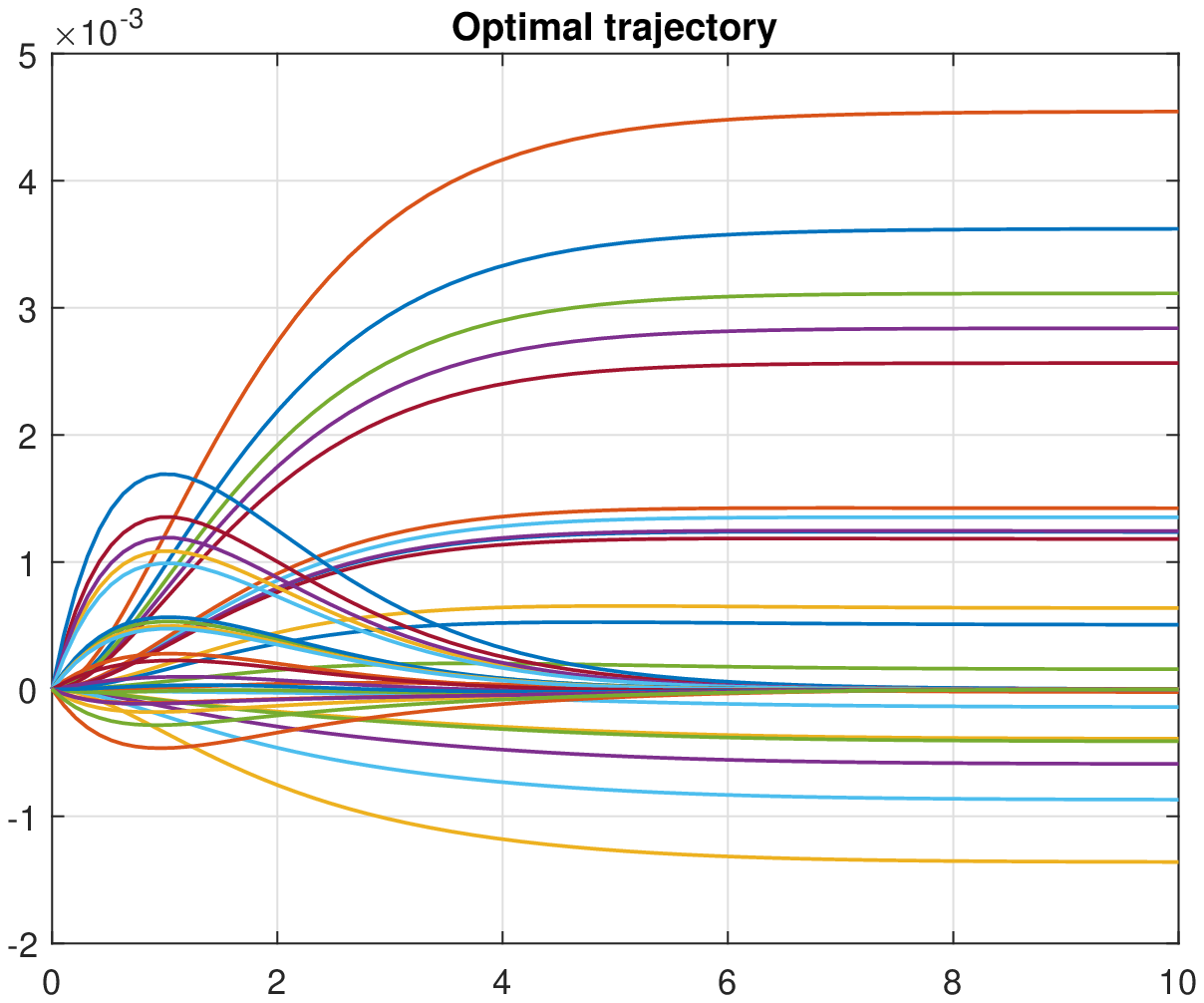}	
	\includegraphics[scale=0.5]{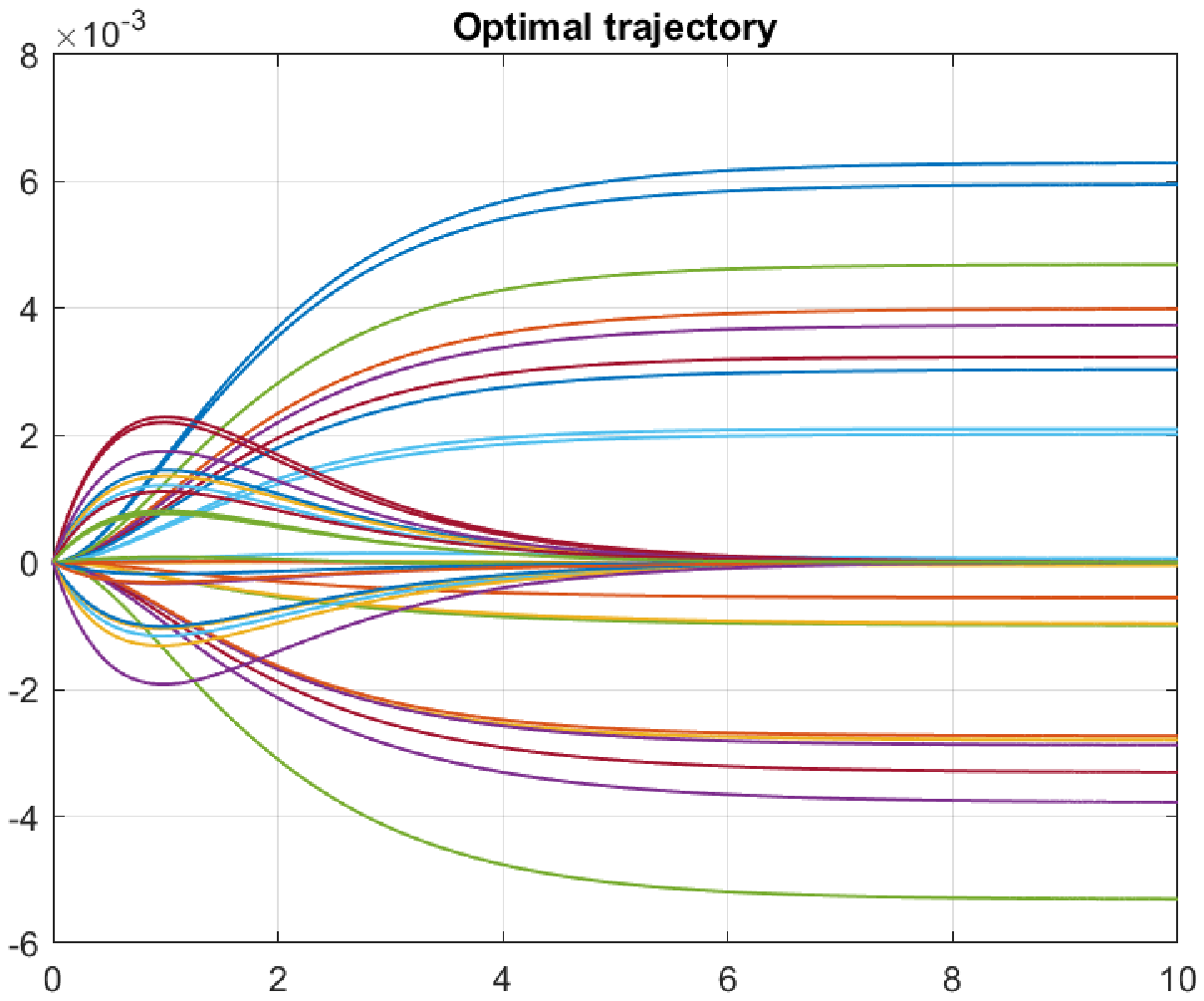}	
\caption{Optimal trajectory for TT (left) and NN (right) with $d=40$ starting from the origin. The spurious nonzero state value at the final time is used to define the switching threshold in the Two Boxes approach.}
\label{Fig_na20_zero}
\end{figure}

To improve the stabilization near the origin, we apply the Two Boxes (TB) approach to reduce the quantity $\tilde y_{\max}(T)$ for both methods. We recall that we are going to consider the sub-domain $[-a_{TB},a_{TB}]^{40}$, where $a_{TB}=2 \tilde y^0_{\max}$ with $\tilde y^0_{\max}$ computed from Figure~\ref{Fig_na20_zero}.
We set $a_{TB}=0.009$ for TT and $a_{TB}=0.0126$ for NN and we apply LQR in the smaller box. The results are shown in Table \ref{Table_comparison_NN}. As discussed previously, without the application of the TB algorithm, TT results are more accurate. Coupling the two methods with TB and LQR, we see by the second line of Table \ref{Table_comparison_NN} a remarkable improvement in terms of the final state magnitude.
% On the other hand, the introduction of the TB approach is deleterious for TT $err_J$ which is affected by an increase.
Finally, we show in Figure \ref{Table_comparison_NN} the optimal trajectories in this case, where the stabilization to the origin is more visible for both approaches.

	\begin{table}[htbp]							
		\centering
		\begin{tabular}{c | c c c c}
					\toprule
				Method	& NN $err_J$ & TT $err_J$ & NN $\tilde y_{\max}(T)$ & TT $\tilde y_{\max}(T)$
					     \\
					\midrule
					Simple & 2.3e-4 & 2.5e-5  & 6.2e-3 & 4.3e-3 \\
					Two Boxes + LQR & 2.0e-4 & 3.7e-5 & 4.8e-4 & 3.3e-4\\

			\bottomrule		
		\end{tabular}	
							\caption{Comparison between NN and TT with $\lambda = 10^{-6}$ for $d=40$. With the Two Boxes technique, both methods are more capable to stabilize the state.\label{Table_comparison_NN}}
						
	\end{table}

	\begin{figure}[htbp]	
\centering
	\includegraphics[scale=0.5]{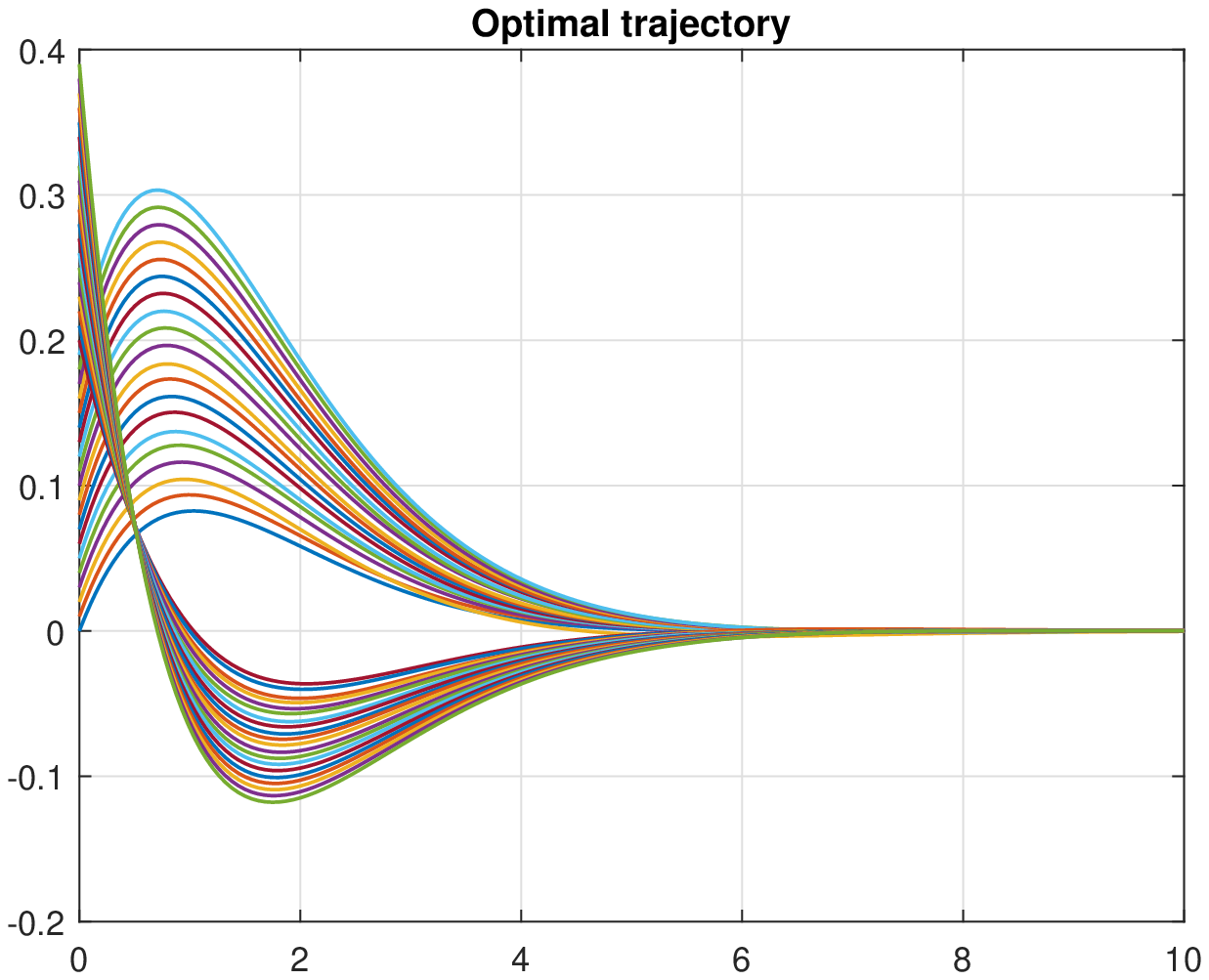}	
	\includegraphics[scale=0.5]{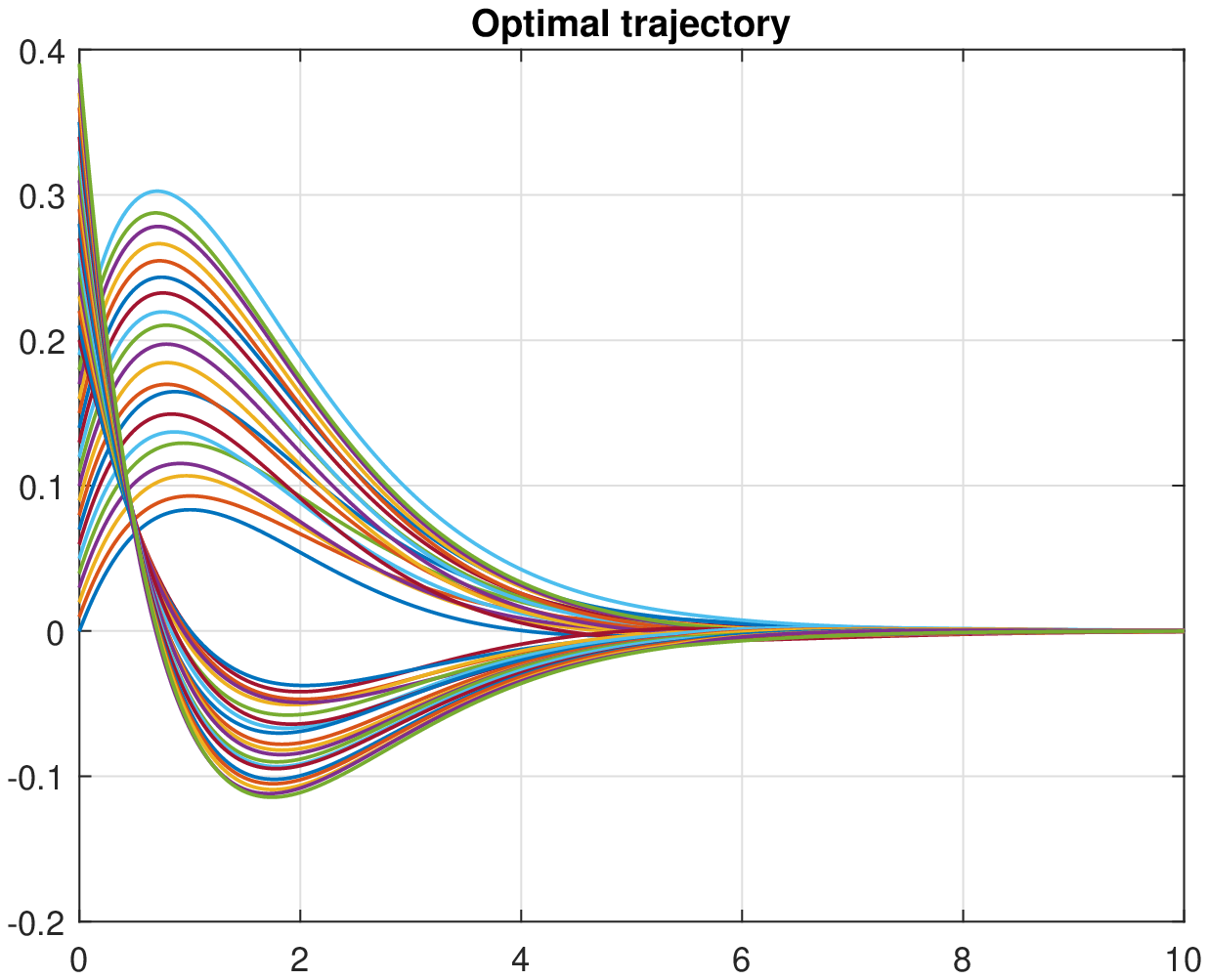}	
\caption{Optimal trajectory for Two Boxes TT (left) and Two Boxes NN (right) with $d=40$. Both methods are stabilizing.}
\label{Fig_TB}
\end{figure}

Finally, we test the TT approach and the Neural Network strategy fixing the dimension $d=100$. For the TT cross we fix $\lambda=10^{-4}$ and the resulting approximation has TT rank equal to 16, in line with the outcome of the left panel of Figure \ref{Fig_shade2}. Table \ref{Table_comparison_NN_100} shows the results of the comparison considering both the simple algorithm and the coupling with the TB approach. We set $a_{TB}=8.6\cdot 10^{-3}$ for TT and $a_{TB}=1.2 \cdot 10^{-2}$ for NN, applying LQR in the smaller box. In this case we note that the TT approximation is two orders of magnitude more accurate than NN in both methods in terms of the error in the cost functional. Furthermore, the application of TB is beneficial for $\tilde y_{\max}(T)$ for both methods, keeping the error in the cost functional unaffected.

	\begin{table}[htbp]							
		\centering
		\begin{tabular}{c | c c c c}
					\toprule
				Method	& NN $err_J$ & TT $err_J$ & NN $\tilde y_{\max}(T)$ & TT $\tilde y_{\max}(T)$
					     \\
					\midrule
					Simple & 2.7e-4
  & 2.8e-6  & 6.0e-3 & 1.5e-3 \\
					Two Boxes + LQR & 2.7e-4   &  2.8e-6 &  4.0e-4 &  4.2e-4\\

			\bottomrule		
		\end{tabular}	
						\caption{Comparison between NN and TT with $\lambda = 10^{-4}$ for $d=100$. The TT approximation achieves an accuracy of two of magnitudes more than NN in terms of error in the cost functional.\label{Table_comparison_NN_100}}
						
	\end{table}

\section{Conclusions}

We have developed a data-driven method for the approximation of high-dimensional infinite horizon optimal control laws. A key feature of the data-driven methodology is that it circumvents the solution of a HJB PDE, a task that quickly becomes overwhelmingly expensive as the dimension of the state space grows. The value function associated to the feedback law has been written in a Functional Tensor Train form and its approximation has been enriched by the knowledge of both the value function and its gradient at specific sampling points. Synthetic data generation has been performed using two different methods: Pontryagin's Maximum Principle and the State Dependent Riccati Equation approach. Through different numerical tests we have shown that the SDRE-based regression performs more accurately and efficiently, whereas PMP can still be necessary in  the case of state/control constraints. The numerical tests have shown that the introduction of gradient-enhanced supervised learning methodology yields the following advantages with respect to the no-gradient formulation:
\begin{itemize}
\item the algorithm presents more stability in presence of noise,
\item the $H^1$ norm of the error is better controlled,
\item improved performance for control-constrained cases,
\item the feedback map appears more regular and can be integrated with larger time steps,
\item it is characterised by a lower standard deviation on the trial set,
\item the $Error/Evals$ cost for different $\theta$ and dimensions is always minimized picking $\lambda \neq 0$.
\end{itemize}

We also showed that the maximum TT rank in the representation of the values function grows almost linearly, yielding a effective mitigation of the curse of dimensionality.

In the future we aim at coupling the proposed algorithm with Model Order Reduction techniques in order to deal with problems in considerably higher dimension such as fluid flow control. Since we are not restricted to consider a reduced space in a very low dimension, we are able to work with challenging problems using an extended reduced order basis, leading to a more accurate control design. Further extensions include the study of robust controllers through differential games and Hamilton-Jacobi-Isaacs PDEs as in \cite{KKK20}, by resorting to representation via SDREs \cite{allasdre,DOLGOV2022510}, and data-driven tensor approximation for stochastic control problems in the spirit of \cite{Han_Jentzen_E_2018}.

\section*{Acknowledgement}
This research was supported by Engineering and Physical Sciences Research Council (EPSRC) grants EP/V04771X/1 and EP/T024429/1.
For the purpose of open access, the author has applied a Creative Commons Attribution (CC-BY) licence to any Author Accepted Manuscript version arising.

\noindent\textbf{Data access.} Matlab codes implementing the gradient cross and numerical examples are available at \url{https://github.com/saluzzi/TT-Gradient-Cross}.

\appendix

\section{Supplementary materials for the multidimensional Gradient Cross}
\label{appendixA}

\subsection{Resolution of the multidimensional least squares problem}
\label{appendixA1}
We report here the computation for the resolution of the regression problem treated in Section \ref{sec:multiTT}.

Deriving the normal equation for the problem \eqref{LSd} we obtain
\begin{equation}
\left(A_< \otimes M \otimes M_> + M_< \otimes A \otimes M_> +  M_< \otimes M \otimes A_> \right) \mathrm{vec}(H^{(k)}) = \vv{F}
\label{normal_eq}
\end{equation}

where
\begin{align*}
A_< & = \sum_{i=1}^{k-1} \lambda_i\left( \partial_i G^{(<k)} \left(\overline{X}_{<k}\right)\right)^{\top} \left( \partial_i G^{(<k)} \left(\overline{X}_{<k}\right)\right), & M_< & =  G^{(<k)} \left(\overline{X}_{<k}\right)^{\top} G^{(<k)} \left(\overline{X}_{<k}\right), \\
A & = \lambda_0 \Phi_k\left(X_k\right)^{\top}\Phi_k\left(X_k\right) + \lambda_k \Phi'_k\left(X_k\right)^{\top}\Phi'_k\left(X_k\right), & M &=\Phi_k\left(X_k\right)^{\top} \Phi_k\left(X_k\right), \\
A_< &= \sum_{i=k+1}^{d} \lambda_i\left( \partial_i G^{(>k)} \left(\overline{X}_{>k}\right)\right)^{\top} \left( \partial_i G^{(>k)} \left(\overline{X}_{>k}\right)\right), & M_> & =  G^{(>k)} \left(\overline{X}_{>k}\right)^{\top} G^{(>k)} \left(\overline{X}_{>k}\right),
\end{align*}

\begin{align*}
\vv{F}& = \sum_{i=1}^{k-1} \lambda_i \left[\partial_i G^{(<k)} \left(\overline{X}_{<k}\right)^{\top} \otimes \Phi^{\top}_k \otimes G^{(>k)}\left(\overline{X}_{>k}\right)^{\top} \right] \vv{V}^k_i+ \lambda_0 \left[ G^{(<k)} \left(\overline{X}_{<k}\right)^{\top} \otimes \Phi^{\top}_k \otimes G^{(>k)}\left(\overline{X}_{>k}\right)^{\top} \right]\vv{V}^k_0 \\
& + \lambda_k \left[ G^{(<k)} \left(\overline{X}_{<k}\right)^{\top} \otimes \Phi'^{\top}_k \otimes G^{(>k)}\left(\overline{X}_{>k}\right)^{\top} \right]\vv{V}^k_k+  \sum_{i=k+1}^{d} \lambda_i \left[ G^{(<k)} \left(\overline{X}_{<k}\right)^{\top} \otimes \Phi^{\top}_k \otimes \partial_i G^{(>k)}\left(\overline{X}_{>k}\right)^{\top} \right] \vv{V}^k_i .
\end{align*}

To solve the normal equation \eqref{normal_eq} we first compute a generalized diagonalization for the three couples of Gram matrices $(M_<,A_<)$, $(M,A)$ and $(M_>,A_>)$,
$$
M_<V_<=A_< V_< L_<\,, \quad MV=A V L\,, \quad  M_>V_>=A_> V_> L_>\,,
$$
then equation \eqref{normal_eq} can be rewritten in the following form 

\begin{equation}
\left(A_< \otimes A \otimes A_> \right) \left(V_< \otimes V \otimes V_> \right) \cdot L_3  \cdot \left(V^{\top}_< \otimes V^{\top} \otimes V^{\top}_> \right) \left(A_< \otimes A \otimes A_> \right)  \mathrm{vec}(H^{(k)}) = \vv{F}
\label{normal_eq2}
\end{equation}

where $L_3=\left(L_< \otimes L \otimes \mathtt{I} + L_< \otimes \mathtt{I} \otimes L_>+\mathtt{I} \otimes L \otimes L_> \right)$ is diagonal. Now the linear system \eqref{normal_eq2} is easily solvable by inverting individual terms under the Kronecker products, and the diagonal matrix.

\subsection{Update of the core evaluations}
\label{appendixA2}

As stated in Section \ref{sec:multiTT}, in the $(k+1)$-th step we can evaluate the core $G^{(<k+1)}\left(\overline{X}_{<k+1}\right)$ with a cost independent of $k$ and $d$.
Indeed, for each of the $r_k$ elements in $I_k^{\alpha}$ and $I_k^x$ we compute the matrix products

\begin{align}\label{eq:G_k+1}
G^{(<k+1)}\left(\overline{X}_{<k+1}(\alpha_k)\right) &= G^{(<k)}\left(\overline{X}_{<k}(I_k^{\alpha}(\alpha_k))\right) G^{(k)}\left(X_{k}(I_k^x(\alpha_k))\right), \\\nonumber
\partial_i G^{(<k+1)}\left(\overline{X}_{<k+1}(\alpha_k)\right) & =  \partial_i G^{(<k)}\left(\overline{X}_{<k}(I_k^{\alpha}(\alpha_k))\right)  G^{(k)}\left(X_{k}(I_k^x(\alpha_k))\right), &  i& =1,\ldots,k-1, \\\nonumber
\partial_k G^{(<k+1)}\left(\overline{X}_{<k+1}(\alpha_k)\right) & = G^{(<k)}\left(\overline{X}_{<k}(I_k^{\alpha}(\alpha_k))\right)  \partial_k  G^{(k)}\left(X_{k}(I_k^x(\alpha_k))\right), & \alpha_k & =1,\ldots,r_k,
\end{align}
where $G^{(<k)}\left(\overline{X}_{<k}\right)$ and $\partial_i G^{(<k)}\left(\overline{X}_{<k}\right)$ are available from the previous step, and need only slicing at $I_k^{\alpha}(\alpha_k)$.

\bibliography{biblio}
\bibliographystyle{abbrv}

\end{document}